\documentclass[11pt]{article}
\usepackage{amsmath}
\usepackage{amsfonts} \makeatother
\usepackage{amsfonts,amsmath,amssymb,mathrsfs}
\usepackage{cite}
\usepackage{amsthm}

\numberwithin{equation}{section}
\newtheorem{theorem}{Theorem}[section]
\newtheorem{lemma}[theorem]{Lemma}
\newtheorem{proposition}[theorem]{Proposition}
\newtheorem{remark}{Remark}[section]

\newtheorem{corollary}[theorem]{Corollary}
\newtheorem{example}{Example}[section]

\newcommand{\be}{\begin{equation}}
\newcommand{\ee}{\end{equation}}
\newcommand\bes{\begin{eqnarray}}
\newcommand\ees{\end{eqnarray}}
\newcommand{\bess}{\begin{eqnarray*}}
\newcommand{\eess}{\end{eqnarray*}}

\newcommand\s{s}
\renewcommand{\div}{\textnormal{div}}

\newcommand{\Fix}{\textnormal{Fix}}

\newcommand\eps{\varepsilon}

\def\R{\mathbb{R}}
\def\N{\mathbb{N}}

\begin{document}

\begin{center} {\bf\Large Ground states and high energy solutions of the

                    planar Schr\"{o}dinger-Poisson system}
\\[4mm]
 {\large Miao Du$^{\text{a}}$, \ \ Tobias Weth$^{\text{b},\ast}$}\\[2mm]
{

$^{\textrm{a}}$ School of Mathematical Sciences, Nanjing Normal University,

     Nanjing 210023, P.R. China

$^{\textrm{b}}$  Institut f\"{u}r Mathematik, Goethe-Universit\"{a}t Frankfurt,

     D-60629 Frankfurt am Main, Germany

}
\renewcommand{\thefootnote}{}
\footnote{\hspace{-2ex}$^{\ast}$ Corresponding author}
\footnote{\hspace{-2ex}\emph {~E-mail addresses:}
 dumiaomath@163.com, weth@math.uni-frankfurt.de.}
\end{center}

\setlength{\baselineskip}{16pt}

\begin{quote}
  \noindent {\bf Abstract:} In this paper, we are concerned with
    the Schr\"{o}dinger-Poisson system
    \begin{equation}\label{0-1}
       \begin{cases}
          -\Delta u + u +\phi u  = |u|^{p-2}u
            &\text{in}\ \mathbb{R}^{d},\\
          \Delta \phi= u^{2} &\text{in}\ \mathbb{R}^{d}.
       \end{cases}
    \end{equation}
Due to its relevance in physics, the system has been extensively studied and is quite well understood in the case $d \ge 3$. In contrast, much less information is available in the planar case $d=2$ which is the focus of the present paper. It has been observed by Cingolani and the second author \cite{Cingolani-Weth-2016} that the variational structure of (\ref{0-1}) differs substantially in the case $d=2$ and leads to a richer structure of the set of solutions. However, the variational approach of \cite{Cingolani-Weth-2016} is restricted to the case $p \ge 4$ which excludes some physically relevant exponents. In the present paper, we remove this unpleasant restriction and explore the more complicated underlying functional geometry in the case $2<p<4$ with a different variational approach.


\end{quote}

\section{Introduction}

$ $
\indent
The present paper is devoted to standing (or solitary) wave
solutions of Schr\"{o}dinger-Poisson systems of the type
\begin{equation}\label{1-2}
   \begin{cases}
      i\psi_{t}- \Delta \psi + W(x)\psi + m \phi \psi = |\psi|^{p-2}\psi
        & \text{in}\ \mathbb{R}^{d}\times \mathbb{R},\\
      \Delta \phi= |\psi|^{2}   & \text{in}\ \mathbb{R}^{d},
   \end{cases}
\end{equation}
where $\psi:\mathbb{R}^{d}\times\mathbb{R}\rightarrow\mathbb{C}$
is the time-dependent wave function, $W:\mathbb{R}^{d}\rightarrow
\mathbb{R}$ is a real external potential, $m \in\mathbb{R}$
is a parameter and $2<p<2^{*}$. Here, $2^{\ast}$ is the so-called
critical Sobolev exponent, that is, $2^{\ast}=\frac{2d}{d-2}$
in case $d\geq3$ and  $2^{\ast}=\infty$ in case $d=1,2$.
Evolution problems of the type (\ref{1-2}) arise in many problems from physics. We refer the reader e.g. to \cite{Mauser-2001}, where (\ref{1-2}) is discussed in a quantum mechanical context where the particular exponent $p=2+\frac{2}{d}$ appears in the case $d \le 3$, see \cite[p. 761]{Mauser-2001}. The function $\phi$ represents an internal potential for a nonlocal self-interaction of the wave function $\psi$.
The usual ansatz
\begin{equation*}
   \psi(x,t)=e^{-i\omega t}u(x)\quad\text{with}
   \ \omega \in \mathbb{R}
\end{equation*}
for standing wave solutions of (\ref{1-2}) leads to the Schr{\"o}dinger-Poisson system
\begin{equation}\label{1-3}
   \begin{cases}
      -\Delta u+V(x)u+ m \phi u=|u|^{p-2}u
        &\text{in}\ \mathbb{R}^{d},\\
      \Delta \phi= u^{2}   &\text{in}\ \mathbb{R}^{d},
   \end{cases}
\end{equation}
where $V(x)=W(x)+\omega$. The second equation in \eqref{1-3} determines
$\phi: \mathbb{R}^{d}\rightarrow\mathbb{R}$ only up to harmonic
functions. It is natural to choose $\phi$ as the negative Newton potential
of $u^{2}$, that is, the convolution of $u^{2}$ with the fundamental
solution $\Phi_{d}$ of the Laplacian, which is given by $\Phi_{d}(x)= -\frac{1}{d(d-2)\omega_{d}}|x|^{2-d}$ in case $d\geq3$ and $\Phi_{d}(x)=\frac{1}{2 \pi} \log |x|$ in case $d=2$. Here $\omega_{d}$ denotes the volume of the unit ball in $\mathbb{R}^{d}$. With this formal inversion of
the second equation in \eqref{1-3}, we obtain the integro-differential
equation
\begin{equation}\label{1-4}
   -\Delta u+V(x)u+ m \left(\Phi_{d}\ast|u|^{2}\right)u =|u|^{p-2}u
   \quad \text{in}\ \mathbb{R}^{d}.
\end{equation}
In the three dimensional case, equation (\ref{1-4}) and its generalizations have been widely studied in recent years, whereas existence, nonexistence and multiplicity results have been obtained under variant assumptions on $V$ and $f$ via variational methods,  see e.g. \cite{Ambrosetti-Ruiz-2008,Azzollini-Pomponio-2008,Bellazzini-Jeanjean-Luo-2008,
Cerami-2010,DAprile-2004-1,DAprile-2005,Ruiz-2006,Wang-Zhou-2007,Zhaoleiga-2013,
Zhaoleiga-2008} and the references therein.

We note that,  at least formally,
\eqref{1-4} has a variational structure related to the energy functional
\begin{align}
   u &\mapsto I(u)=\label{def-functional} \nonumber\\
&\frac{1}{2}\int_{\mathbb{R}^{d}}\left(|\nabla u|^{2}
     +V(x) u^{2}\right)dx+\frac{m}{4}\int_{\mathbb{R}^{d}}
     \int_{\mathbb{R}^{d}}\Phi_{d}\left(|x-y|\right)u^{2}(x)u^2(y)\,dxdy -\frac{1}{p} \int_{\mathbb{R}^{d}}|u|^p dx.
\end{align}
However, in the case where $d=2$ and $\Phi_{d}(x)= \frac{1}{2 \pi} \log |x|$,  the functional $I$ is not well-defined on $H^{1}(\mathbb{R}^{2})$, and this is one of the reasons why much less is known in the planar case.  Inspired by Stubbe \cite{Stubbe-2008}, Cingolani and Weth \cite{Cingolani-Weth-2016} developed a variational framework for the
equation
\begin{equation}
\label{1.6}
    -\Delta u+ V(x)u+\left(\log\left(|\cdot|\right)\ast|u|^{2}\right)u
     =b|u|^{p-2}u  \quad  \text{in}\ \mathbb{R}^{2}
\end{equation}
within the smaller Hilbert space
\begin{equation*}
    X:=\left\{u \in H^{1}(\mathbb{R}^{2}):\int_{\mathbb{R}^{2}}
    \log\left(1+|x|\right)u^{2}dx < \infty \right\}.
\end{equation*}
In this work,  a positive function $V \in L^{\infty}(\mathbb{R}^{2})$, $b\geq 0$ and
$p\geq4$ are considered. In particular, high energy solutions were detected in \cite{Cingolani-Weth-2016} in a periodic setting where the
corresponding functional is invariant under $\mathbb{Z}^{2}$-translations
and therefore fails to satisfy a global Palais-Smale condition. In the case where the external potential $V$ is a positive constant, they also obtained the existence of nonradial
solutions which have arbitrarily many nodal domains. The key tool in \cite{Cingolani-Weth-2016} is a strong compactness
condition (modulo translation) for Cerami sequences at arbitrary positive energy levels. Such a property fails to hold in higher space dimensions, and it is also not available in the case where $2<p<4$.
More precisely, it is an open question whether general Cerami sequences for $I$ in $X$ are bounded in the case $2<p<4$, and therefore no existence results for (\ref{1.6}) have been available for this case. This gap of information is unpleasant not only from a mathematical point of view but also since, as already remarked above, the case $p=3$ is relevant in 2-dimensional quantum mechanical models, see \cite[p. 761]{Mauser-2001}.

The purpose of this paper is  to fill this gap and to provide a counterpart of the results in \cite{Cingolani-Weth-2016} in the case where $2<p<4$ and $V$ is a positive constant.  In particular, we shall prove the existence of  ground state solutions and infinitely many nontrivial sign-changing solutions for \eqref{1.6} in this case.

At this point, we wish to point out some difficulties of the variational approach in the space $X$. First, while the functional $I$ is translation invariant, the norm of $X$ is not translation invariant. Second,
the quadratic part of $I$ is not coercive on $X$. These difficulties have been overcome in \cite{Cingolani-Weth-2016}, and we will partly follow the approach developed there.
The key new difficulty in the case where $2<p<4$ is the competing nature of the local and nonlocal superquadratic terms in the functional $I$ and their different behaviour under scaling transformations. In particular, we note that the nonlinearity $u \mapsto f(u):=|u|^{p-2}u$ with $2<p<4$ does not satisfy the Ambrosetti-Rabinowitz type condition
$$
0<\mu\int_{0}^{u} f(s)ds \leq f(u)u \qquad \text{for all $u\neq0$ with some $\mu>4$}
$$
which would readily imply the boundedness of Palais-Smale sequences. Moreover, the fact that the function $f(s)/|s|^{3}$ is not increasing on $\R \setminus \{0\}$ prevents
us from using Nehari manifold and fibering methods as e.g. in \cite{Rabinowitz-1992,Szulkin-Weth}.
Moreover, in contrast to the higher dimensional case, the logarithmic convolution term does not have a definite sign on $X$. As a consequence,  the boundedness
of Cerami sequences becomes a major difficulty in the variational setting.

To state our main results, we assume that $V$ in (\ref{1-3}) and (\ref{def-functional}) is constant from now on, and without loss we may assume that $V \equiv 1$. Moreover, we focus on (\ref{1-3}) in the case $m>0$, and by rescaling we may assume $m = 2\pi$. Consequently, we are dealing with system (\ref{1-3}), the associated scalar equation
\begin{equation}
\label{1-5}
    -\Delta u+ u+\left(\log\left(|\cdot|\right)\ast|u|^{2}\right)u
     =|u|^{p-2}u  \quad  \text{in}\ \mathbb{R}^{2}
\end{equation}
and the associated energy functional $I: X \to \R$ defined by
\begin{equation}
\label{def-functional-1}
 I(u)=\frac{1}{2}\int_{\mathbb{R}^{2}}\left(|\nabla u|^{2}
     +u^{2}\right)dx + \frac{1}{4}\int_{\mathbb{R}^{2}}
     \int_{\mathbb{R}^{2}}\log \left(|x-y|\right)u^{2}(x)u^2(y)\,dxdy -\frac{1}{p} \int_{\mathbb{R}^{2}}|u|^p dx.
\end{equation}
In the following, by a solution of \eqref{1-5} we always mean a weak solution, i.e. a critical point of $I$.

\begin{remark}
\label{sec:regularity}
{\rm It has been proved in \cite[Proposition 2.3]{Cingolani-Weth-2016} that every critical point $u \in X$ of $I$ is a classical solution of class $C^2$ satisfying $u(x)= o(e^{-\alpha |x|})$ as $|x| \to \infty$ for any $\alpha>0$. Moreover, the function $x \mapsto
w(x)=\int_{\R^2}\log|x-y| u^2(y)\,dy$ is of class $C^3$ on $\R^2$ and satisfies
\begin{equation}
  \label{eq:asymptotics-w}
\Delta w = 2 \pi u^2 \quad \text{in $\R^2$,}\qquad \quad w(x) - \log |x| \int_{\R^2}u^2\,dy
\to 0 \quad \text{as $|x| \to \infty$.}
\end{equation}
It has been assumed that $p \ge 4$ in \cite{Cingolani-Weth-2016}, but the proof of \cite[Proposition 2.3]{Cingolani-Weth-2016} carries over to the case $p \ge 2$ without change.}
\end{remark}

Our first main result is concerned with the existence of mountain pass and ground state solutions. For this we
define the mountain pass value
\begin{equation}
  \label{eq:def-c_mp}
    c_{mp}=\inf_{\gamma \in \Gamma}\max_{t \in [0,1]}I(\gamma(t))\qquad
    \text{with}\qquad \Gamma=\left\{\gamma\in C\left([0,1],\: X\right)
    \ | \ \gamma (0)=0, \ I(\gamma(1))<0\right\}.
\end{equation}

\begin{theorem}\label{th1-1}
Suppose that $p>2$.
\begin{enumerate}
\item[\rm(i)] We have $c_{mp}>0$, and  equation \eqref{1-5} has a solution $u \in X \setminus \{0\}$ with $I(u)=c_{mp}$.
\item[\rm(ii)] Equation \eqref{1-5} has a ground state solution, i.e., a solution $u \in X \setminus \{0\}$ such that $I(u)$ equals the ground state energy
\begin{equation}
\label{def-ground-state-energy}
    c_g:=\inf\left\{I(v)\::\: v\in X\setminus\{0\}\ \text{is\ a\ solution\ of}
    \ \eqref{1-5}\right\}.
\end{equation}
\end{enumerate}
\end{theorem}

In case $2<p<3$, we do not know if the mountain pass energy $c_{mp}$ coincides with the ground state energy $c_g$. In the case where $p \ge 3$, we can derive more information. First, we can identify $c_g$ with $c_{mp}$, and we can give a natural characterization of the ground state energy as a constrained minimum and as a simple minimax value. Second, we will see that every ground state solution does not change sign.
For this purpose, inspired by \cite{Ruiz-2006}, we introduce the auxiliary functional $J: X \to \R$ defined by
\begin{align}
J(u)&=\int_{\R^2} \Bigl(2 |\nabla u|^2+ |u|^2 -\frac{2(p-1)}{p}|u|^{p}\Bigr)dx \label{eq:def-G}\\
&+ \int_{\mathbb{R}^{2}}\int_{\mathbb{R}^{2}}\log\left(|x-y|\right)u^{2}(x)u^{2}(y)dxdy-\frac{1}{4} \left(\int_{\mathbb{R}^{2}}|u|^{2}dx\right)^{2},\nonumber
\end{align}
and the set
\begin{equation}
  \label{eq:def-M}
\mathcal M:= \{u \in X \setminus \{0\}\::\: J(u)=0\}.
\end{equation}
It then follows in a standard way from a Pohozaev type identity given in Lemma~\ref{lem2-4} below that every solution of (\ref{1-5}) is contained in $\mathcal M$. Consequently, the minimal energy value
\begin{equation}
  \label{eq:def-c_M}
c_{\mathcal M}:= \inf_{u \in \mathcal M}I(u).
\end{equation}
on $\mathcal M$ satisfies
\begin{equation}
  \label{eq:eq-comp-c-g-c-M}
c_{\mathcal M} \le c_g \le c_{mp},
\end{equation}
where the latter inequality follows from Theorem~\ref{th1-1}(i). We also define the minimax value
\begin{equation}
  \label{eq:def-c_mm}
c_{mm}:= \inf_{u \in X \setminus \{0\}} \sup_{t > 0} I(u_t).
\end{equation}
where $u_t \in X$ is defined by $u_t(x):= t^2u(tx)$ for $u \in X$, $t > 0$.

\begin{theorem}\label{th1-1-1}
Suppose that $p \ge 3$. Then we have
\begin{equation}
\label{eq:equalitity-c}
c_g= c_{\mathcal M}= c_{mm}= c_{mp}.
\end{equation}
Moreover, this value is attained by $I$ on $\mathcal M$, and every function $u \in \mathcal M$ with $I(u)=c_{\mathcal M}$ is a (ground state) solution of \eqref{1-5} and does not change sign.
\end{theorem}

For equation \eqref{1-4} in case $d \ge 3$ with $m>0$, a result corresponding to Theorem~\ref{th1-1} is proved in
\cite[Theorem 3.4]{Wangjun2012}. In this case, the convolution term of the energy functional in \eqref{def-functional} is negative definite, which is of key importance in the proof in \cite{Wangjun2012}. As remarked above, the convolution term in \eqref{def-functional-1} does not have a definite sign. In the planar case where $p \ge 4$, the existence of a ground state solution has been proved in \cite[Theorem 1.1]{Cingolani-Weth-2016}, where also the existence of infinitely many pairs of solutions is established under more general assumptions. As noted before, the proof in \cite{Cingolani-Weth-2016} does not carry over to the case $2<p<4$ due to the possible lack of boundedness of Cerami sequences.

It is instructive to relate the exponent $p$ to the presence of saddle point structures of the functional $I$. In the case $p \ge 4$ studied in \cite{Cingolani-Weth-2016}, $I$ has a rather simple saddle point structure with regard to the fibres $\{t u\::\: t>0\} \subset X$, $u \in X \setminus \{0\}$, see \cite[Lemma 2.5]{Cingolani-Weth-2016}. In particular, this implies that the ground state energy coincides with the minimum of $I$ on the associated Nehari manifold and obeys the simple minimax
characterization $c_g= \inf \limits_{u \in X \setminus \{0\}} \sup \limits_{t > 0} I(tu)$, see \cite[Theorem 1.1]{Cingolani-Weth-2016}.
In case $2<p<4$ this property is lost, but for $p \ge 3$ we recover a different saddle point structure with respect to the fibres $\{u_t \::\: t>0\} \subset X$, $u \in X \setminus \{0\}$, see Lemma~\ref{lem4-2} below. This new saddle point structure provides the basis for the proof of Theorem~\ref{th1-1-1}. In the case $p \in (2,3)$ we could not find any similar saddle point structure of $I$, and we believe that the more complicated geometry of $I$ in this case is related to particular difficulties regarding the boundedness of Cerami sequences.

Our proof of Theorem \ref{th1-1} uses some preliminary tools from \cite{Cingolani-Weth-2016} and is also inspired by \cite{Jeanjean-1997,Hirata-2010,Moroz-VanSchaftingen-2015}. First we construct a Cerami sequence $(u_{n})_n$
at the mountain-pass level $c_{mp}$ with the extra property
that $J(u_{n})\rightarrow 0$ as $n\rightarrow\infty$.
From this extra information, we then deduce the
boundedness of $(u_{n})_n$ in $H^{1}(\mathbb{R}^{2})$. In the case $p \ge 3$ this step is not difficult, whereas for $2<p<3$ the argument is more subtle. Once this step is taken, we can follow arguments
in \cite{Cingolani-Weth-2016} to pass to a subsequence which -- after suitable translation -- converges to a nontrivial solution $\tilde u$
of \eqref{1-5} with $I(\tilde u)=c_{mp}$. Hence the set $\mathcal K$ of nontrivial solutions of (\ref{1-5}) is nonempty. We then consider a sequence $(u_n)_n$ in $\mathcal K$ with $I(u_n) \to c_g$ as $n \to \infty$, and in the same way as before we may then pass to a subsequence which -- after suitable translation -- converges to a nontrivial solution $u$
of \eqref{1-5} with $I(u)=c_g$.

Our third main result is concerned with a symmetric setting with respect to
suitably defined actions of subgroups of the orthogonal group $O(2)$, and it will give rise to the existence of infinitely many
nonradial sign-changing solutions of \eqref{1-5}. We need to introduce some notation.
Let $G$ be a  closed subgroup of the orthogonal group $O(2)$.
Moreover, let $\tau: G\rightarrow \{-1, 1\}$ be a group homomorphism.
Then the pair $(G, \tau)$ gives rise to a group action of $G$ on $X$
defined by
\begin{equation}\label{1-7}
    [A \ast u] (x):= \tau (A) u \left(A^{-1}x\right)\quad \text{for}
    \ A \in G, \ u \in X\ \text{and}\ x \in \mathbb{R}^{2}.
\end{equation}
The following result is concerned with solutions of \eqref{1-5}
in the invariant space
\begin{equation}
\label{eq:def-X-G}
    X_{G}:=\left\{u \in X\::\: A \ast u = u \ \text{for\ all}
     \ A \in G \right\}.
\end{equation}

\begin{theorem}\label{th1-2}
Suppose that $p>2$. Let $G, \tau$ be as above, assume that
$X_{G} \not = \{0\}$, and let
\begin{equation}
  \label{eq:def-c_gmp}
    c_{mp,G}=\inf_{\gamma \in \Gamma_G}\max_{t \in [0,1]}I(\gamma(t))\quad
    \text{with}\quad \Gamma_G =\left\{\gamma\in C\left([0,1], \:X_G\right)
    \ | \ \gamma (0)=0, \ I(\gamma(1))<0\right\}.
\end{equation}
\begin{enumerate}
\item[\rm(i)] We have $c_{mp,G}>0$, and equation \eqref{1-5} has a solution $u \in X_G \setminus \{0\}$ with $I(u)=c_{mp,G}$.
\item[\rm(ii)] Equation \eqref{1-5} has a $G$-invariant ground state solution $u \in X_G \setminus \{0\}$, i.e., the function $u \in X_{G}$ satisfies $\eqref{1-5}$ and
\begin{equation*}
    I(u)=\inf\left\{I(v)\::\:v\in X_{G}\setminus\{0\}\ \text{is\ a\ solution\ of}
     \ \eqref{1-5}\right\}.
\end{equation*}
\end{enumerate}
\end{theorem}

We remark that if $\tau$ is nontrivial and $A \in G$ is given with $\tau(A)=-1$, then
every $u \in X_{G}$ satisfies $u(A^{-1}x) =-u(x)$. Consequently, we see
that $u$ vanishes on the set $\{x\in \mathbb{R}^{2}\::\:Ax=x\}$ and changes
sign in $\mathbb{R}^{2}$ if $u \neq 0$.
We also point out that Theorem~\ref{th1-2} has no analogue yet in the higher dimensional case, i.e.
for equation \eqref{1-4} in case $d \ge 3$ with $m>0$. This remains an interesting open
problem in the case where $\tau$ is nontrivial.

We discuss some examples for $G$ and $\tau$.

\begin{example}
\label{exam1-1}{\rm
(i) Suppose that $G =O(2)$ and $\tau \equiv 1$. Then the space $X_{G}$ consists
of radial functions in $X$. In this case, Theorem \ref{th1-2} yields the
existence of radial ground state solutions.

(ii) Let $G =\{id,-id, A_{1}, A_{2}\}$, where $A_{i}$ is the reflection
at the coordinate hyperplane $\{x_{i}=0\}$ for $i=1, 2$.
Moreover, let $\tau: G\rightarrow \{-1, 1\}$ be the homomorphism
defined by $\tau(A_{1})=-1$ and $\tau(A_{2}) =1$.
Then $u \in X_{G}$ if and only if
\begin{equation*}
    u(-x_{1}, x_{2})=-u(x_{1}, x_{2}) \quad \text{and} \quad
    u(x_{1}, -x_{2})= u(x_{1}, x_{2})\quad \text{for\ all}\
    x = (x_{1}, x_{2}) \in \mathbb{R}^{2}.
\end{equation*}
Therefore, Theorem \ref{th1-2} yields a $G$-invariant ground state solution
of \eqref{1-5} such that this solution is odd with respect to the hyperplane
$\{x_{1}=0\}$ and even with respect to the hyperplane $\{x_{2}=0\}$.
We point out that this existence result has {\em no} analogue for the seemingly similar nonlinear Schr\"{o}dinger equation
\begin{equation}
  - \Delta u + u=  |u|^{p-2} u, \qquad u \in H^1(\R^2).
\label{eq:nonlinear-schroedinger}
\end{equation}
Indeed, by \cite{esteban-lions82}, (\ref{eq:nonlinear-schroedinger}) does not admit nontrivial solutions which vanish on a hyperplane.

(iii) We assume that, for given $k\in \mathbb{N}$, the subgroup $G$ of $O(2)$
of order $2k$ generated by the (counter-clockwise) $\frac{\pi}{k}$-rotation
\begin{equation*}
    A\in O(2), \quad A x=\left(x_{1}\cos \frac{\pi}{k}-x_{2}
    \sin \frac{\pi}{k},\ x_{1}\sin\frac{\pi}{k}+ x_{2}
    \cos\frac{\pi}{k}\right) \quad \text{for}\
    x=(x_{1}, x_{2}) \in \mathbb{R}^{2}.
\end{equation*}
Let $\tau: G\rightarrow \{-1, 1\}$ be the homomorphism defined by
\begin{equation*}
    \tau(A^{j}) =(-1)^{j} \quad \text{for}\ j=1,\cdot\cdot\cdot,2k,
\end{equation*}
where $A^{j}$ is the $\frac{j\pi}{k}$-rotation. Then Theorem \ref{th1-2}
applies and yields a $G$-invariant  ground state solution.
Note that any such solution is sign-changing and nonradial.
}
\end{example}

From Theorem \ref{th1-2}, applied with group actions of the form given in Example \ref{exam1-1}(iii), we will deduce
the following corollary.

\begin{corollary}\label{coro1-3}
Suppose that $p>2$. Then \eqref{1-5} admits an unbounded sequence $(u_n)_n$ of nonradial sign-changing solutions with $I(u_n) \to \infty$ as $n \to \infty$.
\end{corollary}

This paper is organized as follows. In Section 2, we set up the variational
framework for \eqref{1-5} and present some preliminary results.
In Section 3, we give the proof Theorem \ref{th1-1}.
Section  4 is devoted to the proof of Theorem \ref{th1-1-1}.
Finally, in Section 5 we will complete the proofs of Theorem \ref{th1-2} and Corollary \ref{coro1-3}.

Throughout the paper,  we make use of the following notation. $L^{s}(\mathbb{R}^{2})$, $1\leq s \leq \infty$ denotes the usual Lebesgue space with the norm $|\cdot|_{s}$. For any $\rho>0$ and for any $z\in\mathbb{R}^{2}$, $B_{\rho}(z)$ denotes the ball of radius $\rho$ centered at $z$. As usual, $X'$ denotes the dual space of $X$. Finally,  $C$, $C_{1}$, $C_{2}, \cdot\cdot\cdot$ denote different positive constants whose exact value is inessential.

\section{Preliminaries}

\indent

In the section, we recall the variational setting for \eqref{1-5}
as elaborated by \cite{Cingolani-Weth-2016} and establish some
useful preliminary results. Let $H^{1}(\mathbb{R}^{2})$ be the usual Sobolev space endowed with
the standard inner product
\begin{equation*}
    \langle u, v\rangle =\int_{\mathbb{R}^{2}} \left(\nabla u \nabla v
     + uv\right)dx   \qquad \text{for}\ u, v \in H^{1}(\mathbb{R}^{2}),
\end{equation*}
and the induced norm denoted by $\|u\|=\langle u, u\rangle^{1/2}$.
We define the symmetric bilinear forms
\begin{align*}
    (u,v) \mapsto B_{1}(u,v)&= \int_{\mathbb{R}^{2}}\int_{\mathbb{R}^{2}}
     \log \left(1+|x-y|\right)u(x)v(y)dxdy,\\
    (u,v) \mapsto B_{2}(u,v)&= \int_{\mathbb{R}^{2}}\int_{\mathbb{R}^{2}}
     \log \left(1+\frac{1}{|x-y|}\right)u(x)v(y)dxdy,\\
    (u,v) \mapsto B_{0}(u,v)&= B_{1}(u,v)-B_{2}(u,v) = \int_{\mathbb{R}^{2}}
     \int_{\mathbb{R}^{2}}\log \left(|x-y|\right)u(x)v(y)dxdy.
\end{align*}
Here, in each case, the definition is restricted to measurable
functions $u,v: \mathbb{R}^{2}\rightarrow \mathbb{R}$ such that
the corresponding double integral is well-defined in the Lebesgue
sense. We remark that, since $0< \log (1+r)<r$ for $r>0$, from
the Hardy-Littlewood-Sobolev inequality \cite{Lieb-1983}
we deduce that
\begin{equation}\label{2-1}
    \left|B_{2}(u,v)\right| \leq \int_{\mathbb{R}^{2}}
    \int_{\mathbb{R}^{2}}\frac{1}{|x-y|}\left|u(x)v(y)\right|dxdy
    \leq C_{0}|u|_{\frac{4}{3}}|v|_{\frac{4}{3}}  \qquad
    \text{for}\ u,v \in L^{\frac{4}{3}}(\mathbb{R}^{2})
\end{equation}
with a constant $C_{0}>0$. Then we define the functionals
\begin{align*}
   & V_{1}: H^{1}(\mathbb{R}^{2})\rightarrow [0,\infty],
     &V_{1}(u)& =B_{1}(u^{2},u^{2})=
     \int_{\mathbb{R}^{2}}\int_{\mathbb{R}^{2}}
     \log \left(1+|x-y|\right)u^{2}(x)u^{2}(y)dxdy,\\
   & V_{2}: L^{\frac{8}{3}}(\mathbb{R}^{2})\rightarrow [0,\infty),
     &V_{2}(u)&=B_{2}(u^{2},u^{2})=\int_{\mathbb{R}^{2}}
     \int_{\mathbb{R}^{2}}\log \left(1+\frac{1}{|x-y|}\right)
     u^{2}(x)u^{2}(y)dxdy,\\
   & V_{0}: H^{1}(\mathbb{R}^{2})\rightarrow \mathbb{R} \cup \{\infty\},
     &V_{0}(u)&=B_{0}(u^{2},u^{2})=\int_{\mathbb{R}^{2}}
     \int_{\mathbb{R}^{2}}\log \left(|x-y|\right)u^{2}(x)u^{2}(y)dxdy.
\end{align*}
Note that, by \eqref{2-1} we have
\begin{equation}\label{2-2}
    \left|V_{2}(u)\right| \leq C_{0} |u|_{\frac{8}{3}}^{4} \qquad
    \text{for\ all}\ u \in L^{\frac{8}{3}}(\mathbb{R}^{2}),
\end{equation}
thus $V_{2}$ only takes finite values on $L^{\frac{8}{3}}(\mathbb{R}^{2})
\subset H^{1}(\mathbb{R}^{2})$. Next, for any measurable function
$u: \mathbb{R}^{2}\rightarrow \mathbb{R}$, we define
\begin{equation*}
    |u|_{*} := \left( \int_{\mathbb{R}^{2}} \log\left(1+|x|\right)u^{2}
     dx\right)^{1/2}\in [0, \infty].
\end{equation*}
Since
\begin{equation*}
    \log\left(1+|x-y|\right) \leq \log\left(1+|x|+|y|\right)
    \leq \log\left(1+|x|\right) + \log\left(1+|y|\right)
    \qquad \text{for}\ x, y \in \mathbb{R}^{2},
\end{equation*}
we have the estimate
\begin{align}\label{2-3}
    B_{1}(uv, wz) &\leq \int_{\mathbb{R}^{2}}\int_{\mathbb{R}^{2}}
    \left[\log\left(1+|x|\right)+\log\left(1+|y|\right)\right]
    \left|u(x)v(x)\right|\left|w(y)z(y)\right|dxdy \notag\\
    &\leq |u|_{*}|v|_{*}|w|_{2}|z|_{2}+|u|_{2}|v|_{2}|w|_{*}|z|_{*}
    \qquad \text{for}\ u,v,w,z \in L^{2}(\mathbb{R}^{2})
\end{align}
with the conventions $\infty \cdot 0 = 0$ and $\infty \cdot s =\infty$
for $s>0$. Some useful properties of $B_1$ are contained in the following lemmas from \cite{Cingolani-Weth-2016}.

\begin{lemma}[\hspace{-1ex} {\cite[Lemma 2.1]{Cingolani-Weth-2016}}]
\label{lem2-1}
Suppose that $\{u_{n}\}$ is a sequence in $L^{2}(\mathbb{R}^{2})$
such that $u_{n}\rightarrow u\in L^{2}(\mathbb{R}^{2})\backslash\{0\}$
pointwise a.e. in $\mathbb{R}^{2}$. Moreover, let $\{v_{n}\}$ be
a bounded sequence in $L^{2}(\mathbb{R}^{2})$ such that
\begin{equation*}
    \sup_{n \in \mathbb{N}} B_{1}(u_{n}^{2}, v_{n}^{2}) < \infty.
\end{equation*}
\vskip -0.1 true cm\noindent
Then, there exist $n_{0} \in \mathbb{N}$ and $C>0$ such that
$|v_{n}|_{\ast}<C$ for $n \geq n_{0}$.
If moreover $B_{1}(u_{n}^{2}, v_{n}^{2}) \rightarrow 0$ and $|v_{n}|_{2} \rightarrow 0$ as  $n \rightarrow \infty$, then $|v_{n}|_{\ast} \rightarrow 0$ as $n \rightarrow \infty$.
\end{lemma}

\begin{lemma}[\hspace{-1ex} {\cite[Lemma 2.6]{Cingolani-Weth-2016}}]\label{lem2-7}
Let $\{u_{n}\}$, $\{v_{n}\}$ and $\{w_{n}\}$ be bounded sequences
in $X$ such that $u_{n} \rightharpoonup u$ weakly in $X$. Then, for every
$z \in X$, we have $B_{1}(v_{n}w_{n}, z(u_{n}-u)) \rightarrow 0$ as
$n\rightarrow\infty$.
\end{lemma}

In the following, we fix $p > 2$ and  consider the functional $I\::\: H^{1}
(\mathbb{R}^{2}) \rightarrow \mathbb{R} \cup \{\infty\}$ defined by
\begin{equation*}
    I(u) = \frac{1}{2}\|u\|^{2} + \frac{1}{4} V_{0}(u)
     -\frac{1}{p}\int_{\mathbb{R}^{2}}|u|^{p}dx.
\end{equation*}
\vskip -0.1true cm \noindent
We also define the Hilbert space
\begin{equation}
\label{definition-X}
    X:=\left\{u \in H^{1}(\mathbb{R}^{2})\::\: |u|_{*}< \infty\right\}
\end{equation}
\vskip -0.1 true cm\noindent
with the norm given by $\|u\|_{X}:= \sqrt{\|u\|^{2}+|u|^{2}_{*}}$.
Note that, from \eqref{2-3} we see that the restriction of $I$ to
$X$ (also denoted by $I$ in the sequel) only takes finite values
in $\mathbb{R}$. We need the following facts proved in \cite{Cingolani-Weth-2016}.

\begin{lemma}[{\cite[Lemma 2.2]{Cingolani-Weth-2016}}]\label{lem2-2}$ $

\noindent {\;\rm(i)} The space $X$ is compactly embedded in $L^{s}(\mathbb{R}^{2})$
     for all $s \in [2,\infty)$.\\
{~\rm(ii)} The functionals $V_{0}$, $V_{1}$, $V_{2}$ and $I$ are of class
     $C^{1}$ on $X$. Moreover,
     \begin{equation*}
        V'_{i}(u)v=4B_{i}(u^{2}, uv)\quad \text{for}\ u, v \in X \ \text{and}\
        i=0,1,2.
     \end{equation*}
{\rm(iii)} $V_{2}$ is continuously differentiable on
     $L^{\frac{8}{3}}(\mathbb{R}^{2})$.\\
{\rm(iv)} $V_{1}$ is weakly lower semicontinuous on
     $H^{1}(\mathbb{R}^{2})$.\\
{$~~$\rm(v)} $I$ is weakly lower semicontinuous on $X$.\\
{\rm(vi)} $I$ is lower semicontinuous on $H^{1}(\mathbb{R}^{2})$.
\end{lemma}

Next, we provide a Pohozaev type identity for equation \eqref{1-5}.
The strategy of the proof is similar as e.g. in \cite{B-Lions-1983,DAprile-2004-2},
but some differences occur due to the presence of the logarithmic Newtonian potential.
\begin{lemma}[Pohozaev type identity]
\label{lem2-4}
Suppose that $u \in X$ be a weak solution to \eqref{1-5}.
Then we have the following identity:
\begin{equation*}
    P(u):=\int_{\mathbb{R}^{2}}|u|^{2}dx+\int_{\mathbb{R}^{2}}
    \int_{\mathbb{R}^{2}}\log\left(|x-y|\right)u^{2}(x)u^{2}(y)dxdy
    +\frac{1}{4}\left(\int_{\mathbb{R}^{2}}u^{2}dx\right)^{2}
    -\frac{2}{p}\int_{\mathbb{R}^{2}}|u|^{p}dx=0.
\end{equation*}
\end{lemma}

\begin{proof}
We first recall from Remark~\ref{sec:regularity} that $u$ is of class $C^{2}$ with $u(x)= o(-e^{\alpha |x|})$ as $|x | \rightarrow \infty$ for any $\alpha>0$, and that the function
$$
w: \R^2 \to \R, \qquad w(x)=\int_{\R^{2}}\log(|x-y|)u^{2}(y)dy,
$$
is of class $C^3$. In this proof, we also consider the functions
$$
g(u)=|u|^{p-2}u-u \quad \text{and}\quad  G(u)=\int_{0}^{u}g(s)ds= \frac{|u|^p}{p}-\frac{u^2}{2}\qquad \in C^2(\R)
$$
which also decay exponentially as $|x| \to \infty$. In the first part of the proof we use the device of Pohozaev \cite{Pohozave-1965},
multiplying the equation by $x \cdot \nabla u$ and integrating by parts to get an integral identity on a ball $B_{R}(0)$. In the second part, we then show that the
boundary term in the identity tends to zero as $R\rightarrow \infty$. So let $R>0$. Since, as noted e.g. in \cite[p. 136]{Willem-1996}, for any function $u \in C^2(\R^2)$ we have
$$
\Delta u (x \cdot \nabla u) = \div\Bigl (\nabla u (x \cdot \nabla u)- x \frac{|\nabla u|^2}{2}\Bigr) \qquad \text{on $\R^2$,}
$$
the divergence theorem gives
\begin{equation}
    \int_{B_{R}(0)}-\Delta u (x \cdot \nabla u)dx = -\frac{1}{R}\int_{\partial B_{R}(0)}
      |x \cdot \nabla u|^{2}d\sigma + \frac{R}{2} \int_{\partial B_{R}(0)}|\nabla u|^{2}d\sigma. \label{eq:Pohozaev-1}
\end{equation}
Similarly, since $g(u) (x \cdot \nabla u) = \div (x \,G(u))- 2 G(u)$ on $\R^2$, we have
\begin{equation}
    \int_{B_{R}(0)} g(u)(x \cdot \nabla u)dx= -2 \int_{B_{R}(0)}G(u)dx +R\int_{\partial B_{R}(0)}G(u)d\sigma;\label{eq:Pohozaev-2}.
\end{equation}
Moreover, since $w u (x \cdot \nabla u) = \frac{1}{2} \Bigl( \div [x\, w u^2] - u^2 (x \cdot \nabla w)-2w u^2 \Bigr)$, we have
\begin{equation}
    \int_{B_{R}(0)} w u(x \cdot \nabla u)dx= - \frac{1}{2}\int_{B_{R}(0)}u^2 (x \cdot \nabla w)dx
     - \int_{B_{R}(0)}wu^2 dx + \frac{R}{2} \int_{\partial B_{R}(0)}w u^{2}d\sigma.\label{eq:Pohozaev-3}
\end{equation}
Thus, multiplying \eqref{1-5} by $x \cdot \nabla u$ and  integrating on $B_R(0)$, we deduce from
\eqref{eq:Pohozaev-1}, \eqref{eq:Pohozaev-2} and \eqref{eq:Pohozaev-3} that
\begin{equation}
\int_{B_{R}(0)}\Bigl(\frac{u^2 (x \cdot \nabla w)}{2} +w u^2-2 G(u)\Bigr)dx
    = \int_{\partial B_{R}(0)} \Bigl(-\frac{|x \cdot \nabla u|^{2}}{R}+ R \Bigl(\frac{|\nabla u|^{2} }{2} +\frac{w u^{2} }{2}- G(u)\Bigr)\Bigr)
d\sigma.\label{eq:Pohozaev-4}
\end{equation}
Next, following the idea in \cite{B-Lions-1983},
we will show that the right hand side in \eqref{eq:Pohozaev-4} converges
to zero for a suitable sequence $R_n \rightarrow \infty$, i.e.,
\begin{equation}
\label{f-R-N}
    R_n \int_{\partial B_{R_n}(0)} |f|\;d\sigma \rightarrow 0\quad \text{for the function $x \mapsto f(x)=  \frac{|\nabla u|^{2} }{2} +\frac{w u^{2} }{2}- G(u)-\frac{|x \cdot \nabla u|^{2}}{|x|^2}$.}
\end{equation}
For this it suffices to show that
\begin{equation}
  \label{eq:L1-suff}
f \in L^1(\R^{2}).
\end{equation}
Indeed, if no sequence $(R_n)_n$ with $R_n \rightarrow \infty$ and
(\ref{f-R-N}) exists,  it follows that
\begin{equation*}
    \int_{\partial B_{R}(0)} |f|\;d\sigma \geq \frac{c}{R} \quad \text{for $R\geq R_0$ with some constants $c, R_0 >0$.}
\end{equation*}
This contradicts (\ref{eq:L1-suff}), as
\begin{equation*}
    \int_{\R^2}|f|\;dx = \int_{0}^{\infty} d R\int_{\partial B_{R}(0)} |f|\;d\sigma
    \geq c\int_{R_0}^{\infty} \frac{1}{R} \; dR = \infty.
\end{equation*}
To see (\ref{eq:L1-suff}), it suffices to recall from Remark~\ref{sec:regularity} that $u$ and $\nabla u$ decay exponentially, whereas $w$ grows logarithmically as $|x| \to \infty$. Indeed, the exponential decay of $\nabla u$ follows by the decay of $u$ and standard elliptic regularity. Consequently, the function $f$ also decay exponentially as $|x| \to \infty$, which gives (\ref{eq:L1-suff}) and therefore (\ref{f-R-N}).
To conclude the proof, we note that, by the same arguments,
\begin{equation*}
    wu^2 \in L^{1}(\mathbb{R}^{2}) \quad \text{and}\quad   G(u) \in L^{1}(\mathbb{R}^{2}).
\end{equation*}
Moreover a direct calculation gives
$$
x \cdot \nabla w(x) = \int_{\R^{2}}\frac{|x|^{2}- x\cdot y}{|x-y|^2}u^{2}(y)dy \qquad \text{for $x \in \R^2$.}
$$
By the exponential decay of $u$, it then follows that $|x \cdot \nabla w(x)| \le C|x|$ for $x \in \R^2$ with a constant $C>0$, and thus $u^{2}(x \cdot \nabla w) \in L^{1}(\mathbb{R}^{2})$. Moreover,
\begin{align*}
\int_{\R^2}u^2 (x \cdot \nabla w)dx=\int_{\R^{2}} \int_{\R^2} \frac{|x|^{2}- x\cdot y}{|x-y|^2}u^2(x)u^{2}(y)dy &=\frac{1}{2}
\int_{\R^{2}} \int_{\R^2} \frac{|x|^{2}+|y|^2- 2 x\cdot y}{|x-y|^2}u^2(x)u^{2}(y)dy\\
&= \frac{1}{2}\left(\int_{\mathbb{R}^{2}}u^{2}dx\right)^{2}.
\end{align*}
Consequently, with $P(u)$ as given in the statement of the lemma and $f$ given in (\ref{f-R-N}) we have, by (\ref{eq:Pohozaev-4}),
\begin{align*}
P(u) = \int_{\R^2}\Bigl(\frac{u^2 (x \cdot \nabla w)}{2} + w u^2-2 G(u)\Bigr)dx &= \lim_{n \to \infty} \int_{B_{R_n}(0)}\Bigl(\frac{u^2 (x \cdot \nabla w)}{2} w u^2-2 G(u)\Bigr)dx\\
&= \lim_{n \to \infty}  R_n \int_{\partial B_{R_n}(0)} |f|\;d\sigma =  0,
\end{align*}
as claimed.
\end{proof}

We close this section with some observations on the shape of the functional $I$.

\begin{lemma}\label{lem2-5}
There exists $\rho>0$ such that
\begin{equation}\label{2-5}
    m_{\beta} := \inf\{I(u)\::\: u \in X, \ \|u\|=\beta\} >0
    \qquad \text{for}\ 0<\beta \leq \rho
\end{equation}
and
\begin{equation}
\label{2-5-1}
    n_{\beta} :=\inf \{I'(u)u \::\: u \in X\:,\: \|u\|=\beta \}>0 \qquad \text{for $0 <
  \beta \le \rho$.}
\end{equation}

\end{lemma}

\begin{proof}
For each $u \in X$, by \eqref{2-2} and Sobolev embeddings we have
$$
 I(u)\geq \frac{\|u\|^{2}}{2}-\frac{V_{2}(u)}{4}-\frac{1}{p}
       \int_{\mathbb{R}^{2}}|u|^{p}dx \geq \frac{\|u\|^{2}}{2} - \frac{C_{0}}{4}|u|_{\frac{8}{3}}^{4}
       -\frac{1}{p}|u|_{p}^{p} \geq \frac{\|u\|^{2}}{2}\left(1-C_{1}\|u\|^{2}
       -C_{2}\|u\|^{p-2}\right),
$$
where $C_0, C_{1}, C_{2}>0$ are constants.  This implies that
\eqref{2-5} holds for $\rho>0$ sufficiently small. Since
$$
I'(u)(u)=  \|u\|^2 + V_0(u) - |u|_p^p \ge  \|u\|^2- V_2(u)
- |u|_p^p
$$
for $u \in X$, a similar estimate shows that \eqref{2-5-1} holds for
$\rho>0$ sufficiently small.
\end{proof}

\begin{lemma}\label{lem2-6}
Let $u \in X \setminus \{0\}$, and let $u_t \in X$ be defined by $u_t(x):= t^2u(tx)$ for $x \in \mathbb R^2$, $t > 0$.
Then we have
$$
I(u_t) \to -\infty \qquad \text{as $t \to \infty$.}
$$
In particular, the functional $I$ is not bounded from below.
\end{lemma}

\begin{proof}
Let $u\in X \backslash \{0\}$. Then we have
\begin{align*}
  I(u_{t})&=\frac{t^{4}}{2}\int_{\mathbb{R}^{2}}|\nabla u|^{2}dx
    +\frac{t^{2}}{2}\int_{\mathbb{R}^{2}}u^{2}dx+\frac{t^{4}}{4}
    \int_{\mathbb{R}^{2}}\int_{\mathbb{R}^{2}}\log\left(|x-y|\right)
    u^{2}(x)u^{2}(y)dxdy\\
  &\quad-\frac{t^{4}\log t}{4}\left(\int_{\mathbb{R}^{2}}|u|^{2}dx
    \right)^{2}-\frac{t^{2p-2}}{p}\int_{\mathbb{R}^{2}}|u|^{p}dx.
\end{align*}
Consequently, $I(u_{t})\rightarrow -\infty$ as $t\rightarrow\infty$, and the claim follows.
\end{proof}

\section{Existence of mountain pass and ground state solutions to \eqref{1-5}}
\label{sec:exist-ground-state}

\indent

In this section, we will prove Theorem~\ref{th1-1}. For this we will first prove the existence of critical points of $I$ at the mountain pass energy level $c_{mp}$ defined in (\ref{eq:def-c_mp}). Within this step, we shall use the following
general minimax principle from \cite{li-wang:11}. It is a somewhat stronger variant of \cite[Theorem 2.8]{Willem-1996} which gives rise to Cerami sequences instead of Palais-Smale sequences.

\begin{proposition}[{$\!\!$ \cite[Proposition 2.8]{li-wang:11}}]
\label{prop3-2}
Let $X$ be a Banach space. Let $M_{0}$ be a closed subspace of the metric
space $M$ and $\Gamma_{0} \subset C(M_{0},\: X)$. Define
\vskip -0.2 true cm
\begin{equation*}
    \Gamma := \left\{\gamma \in C(M,\: X)\::\: \gamma|_{M_{0}} \in \Gamma_{0}
    \right\}.
\end{equation*}
\vskip -0.1 true cm
\noindent If $\varphi \in C^{1}(X,\: \mathbb{R})$ satisfies
\begin{equation*}
    \infty > c := \inf_{\gamma \in \Gamma} \sup_{u \in M} \varphi\left(
    \gamma (u)\right) > a:= \sup_{\gamma_{0}\in \Gamma_{0}} \sup_{u \in M_{0}}
    \varphi\left(\gamma_{0}(u)\right),
\end{equation*}
then, for every $\varepsilon \in \left(0,\frac{c-a}{2}\right)$, $\delta>0$ and
$\gamma \in \Gamma $ with $\sup \limits_{u \in M}\varphi(\gamma(u)) \leq c+\varepsilon$ there exists $u \in  X$ such that
\begin{itemize}
  \item [\rm(a)] $c - 2\varepsilon \leq \varphi(u) \leq c + 2\varepsilon$,
\vspace{-0.2cm}

  \item [\rm(b)] $\text{dist}\ (u, \gamma(M)) \leq 2\delta$,
\vspace{-0.2cm}

  \item [\rm(c)] $\left(1+\|u\|_{X}\right)\|\varphi'(u)\|_{X'} \leq
     \frac{8\varepsilon}{\delta}$.
\end{itemize}
\end{proposition}

We now consider the mountain pass value
\begin{equation*}
c_{mp}=\inf\limits_{\gamma \in \Gamma}
   \max_{t \in [0,1]}I(\gamma(t)),
\end{equation*}
where
$$
\Gamma=\left\{\gamma\in C\left([0,1],\: X\right)\ | \ \gamma(0)=0, \ I(\gamma(1))<0\right\}.
$$
By Lemmas \ref{lem2-5} and \ref{lem2-6}, we find that
$$
0 < m_{\rho} \leq c_{mp}< \infty,
$$
which means that the functional $I$ has a mountain pass geometry. As the following lemma shows, we can now use Proposition \ref{prop3-2}
to prove the existence of a Cerami sequence $\{u_{n}\}\subset X$ at the energy level $c_{mp}$ with a key additional property. For Palais-Smale sequences in related variational settings, this idea goes back to \cite{Jeanjean-1997} and
has also been used in \cite{Hirata-2010,Moroz-VanSchaftingen-2015}.
\begin{lemma}\label{lem3-3}
Let $p >2 $. Then there exists a sequence $\{u_{n}\}$ in $X$ such that,
as $n\rightarrow\infty$,
\begin{equation}\label{3-2}
  I(u_{n})\rightarrow c_{mp},
  \quad \|I'(u_{n})\|_{X'}\left(1+\|u_{n}\|_{X}\right)\rightarrow 0
  \quad \text{and}\quad J(u_{n}) \rightarrow 0,
\end{equation}
where $J: X \to \R$ is defined by (\ref{eq:def-G}).
\end{lemma}

\begin{proof}
Following the strategy of \cite{Jeanjean-1997,Hirata-2010,Moroz-VanSchaftingen-2015}, we consider the Banach space
$$
\tilde X := \mathbb{R}\times X
$$
equipped with the standard product norm
$\|(\s, v)\|_{\tilde X}:=\left(|\s|^{2}+\|v\|_{X}
\right)^{1/2}$ for $\s \in\mathbb{R}$, $v \in X$. Moreover, we
define the continuous map
$$
\rho:\tilde X \rightarrow X, \qquad \rho(\s, v)[x] :=e^{2\s}
v\left(e^{\s}x\right) \qquad \text{for $\s \in\mathbb{R}$, $v \in X$ and $x \in \mathbb{R}^{2}$.}
$$
We also consider the functional
$$
\varphi:= I \circ \rho\::\: \tilde X \to \R.
$$
A short computation yields that
\begin{align}
  \varphi(\s, v)= I(\rho(\s, v))&=\frac{e^{4\s}}{2}\int_{\mathbb{R}^{2}}
    |\nabla v|^{2}dx+\frac{e^{2\s}}{2}\int_{\mathbb{R}^{2}}v^{2}dx
    +\frac{e^{4\s}}{4}\int_{\mathbb{R}^{2}}\int_{\mathbb{R}^{2}}
    \log\left(|x-y|\right)v^{2}(x)v^{2}(y)dxdy \nonumber\\
  &\quad-\frac{\s e^{4\s}}{4}\left(\int_{\mathbb{R}^{2}}|v|^{2}dx\right)^{2}
    -\frac{e^{2\s(p-1)}}{p}\int_{\mathbb{R}^{2}}|v|^{p}dx \qquad \text{for $\s \in\mathbb{R}$ and $v \in X$.} \label{comp-rho-s-v}
\end{align}
This readily implies that $\varphi$ is of class $C^1$ on $X$ with
\begin{align}
\partial_{\s}  \varphi(\s, v) &=2 e^{4\s}\int_{\mathbb{R}^{2}}
    |\nabla v|^{2}dx+ e^{2\s}\int_{\mathbb{R}^{2}}v^{2}dx
    +e^{4\s}\int_{\mathbb{R}^{2}}\int_{\mathbb{R}^{2}}
    \log\left(|x-y|\right)v^{2}(x)v^{2}(y)dxdy\nonumber\\
  &\quad-\Bigl(\s e^{4\s} + \frac{e^{4\s}}{4}\Bigr) \left(\int_{\mathbb{R}^{2}}|v|^{2}dx\right)^{2}
    -\frac{2(p-1)}{p}e^{2\s(p-1)}\int_{\mathbb{R}^{2}}|v|^{p}dx \nonumber\\
&=2 \int_{\mathbb{R}^{2}}
    |\nabla \rho(\s, v)|^{2}dx+ \int_{\mathbb{R}^{2}}\rho(\s, v)^{2}dx
    +\int_{\mathbb{R}^{2}}\int_{\mathbb{R}^{2}}
    \log\left(|x-y|\right)\rho(\s, v)^{2}(x)\rho(\s, v)^{2}(y)dxdy\nonumber \\
  &\quad-\frac{1}{4} \left(\int_{\mathbb{R}^{2}}|\rho(\s, v)|^{2}dx\right)^{2}
    -\frac{2(p-1)}{p}\int_{\mathbb{R}^{2}}|\rho(\s, v)|^{p}dx \nonumber\\
  &=J(\rho(\s, v)) \qquad \text{for $\s \in\mathbb{R}$ and $v \in X$}, \label{comp-partial-s}
\end{align}
where $J$ defined in (\ref{eq:def-G}). Moreover, since the map $v \mapsto \rho(\s, v)$ is linear for fixed $\s \in \R$, we have
\begin{equation}
  \label{eq:comp-partial-v}
\partial_{v}  \varphi(\s, v) w=I'(\rho(\s, v))\rho(\s, w) \qquad \text{for $\s \in\mathbb{R}$ and $v, w \in X$.}
\end{equation}
Next, we define the minimax value
\begin{equation*}
    c_{\ast}=\inf_{\tilde{\gamma }\in \tilde{\Gamma}}\max_{t \in [0,1]}
    \varphi(\tilde{\gamma}(t)),
\end{equation*}
where
\begin{equation*}
  \tilde{\Gamma}:=\left\{ \tilde{\gamma} \in C([0,1],\:\tilde X)\ |
  \ \tilde{\gamma} (0)=(0,0),\ \varphi(\tilde{\gamma}(1))<0 \right\}.
\end{equation*}
Since $\Gamma =\{\rho \circ \tilde{\gamma}\::\: \tilde{\gamma}\in \tilde{\Gamma}\}$, the
minimax values of $I$ and $\varphi$ coincide, i.e., $c_{mp}=c_{\ast}$. We will now apply Proposition \ref{prop3-2} to the functional $\varphi$, $M= [0,1]$, $M_0= \{0,1\}$ and $\tilde X$, $\tilde \Gamma$ in place of $X$, $\Gamma$. More precisely, for fixed $n \in \mathbb N$
we use the definition of $c_{mp}$ and choose $\gamma_{n}\in\Gamma$ with
\begin{equation*}
    \max_{t\in [0,1]} I(\gamma_{n}(t)) \leq c_{mp}+\frac{1}{n^{2}}.
\end{equation*}
We then define $\tilde \gamma_n \in \tilde \Gamma$ by $\tilde{\gamma}_{n}(t)=(0, \gamma_{n}(t))$,  and we note that
\begin{equation*}
   \max_{t\in[0,1]} \varphi(\tilde{\gamma}_{n}(t))=\max_{t\in[0,1]}
   I(\gamma_{n}(t))\leq c_{mp} + \frac{1}{n^{2}}.
\end{equation*}
An application of Proposition~\ref{prop3-2} with $\tilde \gamma_n$ in place of $\gamma$ and $\varepsilon=\frac{1}{n^{2}}$, $\delta=\frac{1}{n}$ now yields the existence of $(\s_{n},v_{n})\in \tilde X$ such that, as $n \to \infty$,
\begin{align}\label{3-3}
    &\varphi (\s_{n}, v_{n})\rightarrow c_{mp},\\
     &\|\varphi'(\s_{n}, v_{n})\|_{\tilde X'}
     \left(1+\|(\s_{n}, v_{n})\|_{\tilde X}\right)
     \rightarrow 0,\quad \label{3-4}\\
&{\rm dist}\bigl((\s_{n},v_{n}), \:\{0\} \times \gamma_{n}([0,1])\bigr)  \to 0,\label{3-5-0}
\end{align}
whereas (\ref{3-5-0}) obviously implies that
\begin{equation}
\label{3-5}
    \s_{n}\rightarrow 0.
\end{equation}
Since
\begin{equation}\label{3-6}
     \varphi'(\s_{n}, v_{n}) (h,w)
    =I'(\rho (\s_{n}, v_{n}))\rho (\s_{n}, w)
    + J(\rho (\s_{n}, v_{n}))h \qquad \text{for $(h,w) \in \tilde X$}
\end{equation}
by (\ref{comp-partial-s}) and (\ref{eq:comp-partial-v}), we may take $h=1$ and $w=0$ in \eqref{3-6} to obtain
\begin{equation}
\label{eq:G-zero-conv}
    J(\rho (\s_{n}, v_{n}))\rightarrow 0
    \qquad \text{as}\ n \rightarrow \infty.
\end{equation}
For $u_{n}:=\rho (\s_{n}, v_{n})$, it then follows from (\ref{3-3}) and (\ref{eq:G-zero-conv}) that
\begin{equation*}
   I(u_{n})\rightarrow c_{mp} \quad \text{and} \quad
   J(u_{n})\rightarrow 0 \qquad \text{as}\ n \rightarrow \infty.
\end{equation*}
Finally, for given $v \in X$ we consider $w_n =e^{-2 \s_{n}}v(e^{-\s_{n}} \cdot) \in X$
and deduce from (\ref{3-4}) and (\ref{3-6}) with $h=0$ that
\begin{equation*}
    \left(1+\|u_{n}\|_{X}\right)\left|I'(u_{n})v \right|= \left(1+\|u_{n}\|_{X}\right)\left|I'(u_{n})\rho(\s_n,w_n) \right|
= o(1)\|w_n \|_{X} \qquad \text{as $n \to \infty$,}
\end{equation*}
whereas by (\ref{3-5}) we have
\begin{align*}
\|w_n \|_{X}^2 &= \|w_n\|^2 + |w_n|_*^2 \\
&= e^{-4\s_n} \int_{\mathbb{R}^{2}} |\nabla v|^{2}dx+ e^{-2\s_n}\int_{\mathbb{R}^{2}}\Bigl(1+ \log(1+e^{2 \s_n}|x|)\Bigr) v^{2}dx\\
&= [1+o(1)] \int_{\mathbb{R}^{2}} |\nabla v|^{2}dx+ [1+o(1)] \int_{\mathbb{R}^{2}}\Bigl(1+ \log(1+|x|)\Bigr) v^{2}dx \\
&= \bigl(1+o(1)\bigr)\|v\|_{X}^2 \qquad \text{as $n \to \infty$}
\end{align*}
with $o(1) \to 0$ uniformly in $v \in X$. Combining the latter two estimates, we get that
$$
    \left(1+\|u_{n}\|_{X}\right)\|I'(u_{n})\|_{X'} \to 0 \qquad \text{as $n \to \infty$.}
$$
The proof is thus finished.
\end{proof}

In the following key proposition, we shall show, in particular, that any sequence $(u_{n})_n$
satisfying \eqref{3-2} is bounded in $H^{1}(\mathbb{R}^{2})$.

\begin{proposition}\label{lem3-4}
Let $p>2$, and $\{u_{n}\}$ be a sequence in $X$ such that
\begin{equation}\label{3-2-1}
c:= \sup_{n  \in \N}  I(u_{n})<\infty
  \qquad \text{and} \qquad \|I'(u_{n})\|_{X'}\left(1+\|u_{n}\|_{X}\right)\rightarrow 0,
  \quad J(u_{n}) \rightarrow 0 \quad \text{as $n\rightarrow\infty$.}
\end{equation}
Then $(u_{n})_n$ is bounded in $H^{1}(\mathbb{R}^{2})$.
\end{proposition}

\begin{proof}
In the following, $C_1,C_2,\cdots$ denote positive constants independent of $n \in \mathbb N$.
We first observe from \eqref{3-2-1} that
\begin{equation}\label{3-7}
  c+o(1)  \ge I(u_{n})-\frac{1}{4}J(u_{n}) = \frac{1}{4}|u_{n}|_{2}^{2}+\frac{1}{16}|u_{n}|_{2}^{4}
     +\frac{p-3}{2p}|u_{n}|_{p}^{p}.
\end{equation}
We may then distinguish the following two cases:

\emph{Case 1: $p>3$}. In this case, \eqref{3-7} implies that $(u_{n})_n$
is bounded in $L^{2}(\mathbb{R}^{2})$ and in $L^{p}(\mathbb{R}^{2})$.
Therefore, by \eqref{2-2} we have
\begin{equation*}
    V_{2}(u_{n})\leq C_{0}|u_{n}|_{\frac{8}{3}}^{4}
    \leq C_{0}|u_{n}|_{2}^{4(1-\theta_{0})}|u_{n}|_{p}^{4\theta_{0}}
    \leq C_{1},
\end{equation*}
where $\theta_{0}=\frac{p}{4(p-2)}$. Consequently, we may use \eqref{3-2-1} again to estimate
\begin{equation*}
    2\|u_{n}\|^{2}+ V_{1}(u_{n}) = 4I(u_{n}) + V_{2}(u_{n})
       +\frac{4}{p}|u_{n}|_{p}^{p}\leq 4c +C_1 +\frac{4}{p}|u_{n}|_{p}^{p} \le C_{2}.
\end{equation*}
This implies that $\{u_{n}\}$ is bounded in $H^{1}(\mathbb{R}^{2})$.

\emph{Case 2: $2< p \leq 3$}.  We first claim that
\begin{equation}\label{3-8}
    |\nabla u_{n}|_{2} \leq C_3 \qquad \text{for $n \in \mathbb{N}$.}
\end{equation}
Suppose by contradiction that this is false.
Then, after passing to a subsequence, we have
\begin{equation}\label{3-10}
    |\nabla u_{n}|_{2}\rightarrow \infty\qquad \text{as}\
    n \rightarrow \infty.
\end{equation}
Let $t_{n}:=|\nabla u_{n}|_{2}^{-1/2}$ for $n \in \mathbb{N}$, so that $t_{n}\rightarrow0$ as $n\rightarrow\infty$.
For $n \in \mathbb N$ we define the rescaled functions $v_n \in X$ by $v_{n}(x):=t_{n}^{2}u_{n}(t_{n}x)$ for $n \in \mathbb{N}$,
so that
\begin{equation}
  \label{eq:normalized}
|\nabla v_{n}|_{2}=1 \quad \text{and}\quad  |v_n|_q^q = t_n^{2q-2}|u_n|_q^q \qquad \text{for all $n \in \mathbb N$, $1 \le q < \infty$.}
\end{equation}
By the Gagliardo-Nirenberg inequality,
\begin{equation}\label{3-11}
    |v_{n}|_{p}^{p}\leq C_{4} |v_{n}|_{2}^{2}|\nabla v_{n}|_{2}^{p-2}= C_4 |v_{n}|_{2}^{2} \qquad \text{for $n \in \mathbb N$}.
\end{equation}
Multiplying \eqref{3-7} by $t_{n}^{4}$, we deduce, from (\ref{eq:normalized}) and \eqref{3-11},
\begin{align*}
   c t_{n}^{4}+o(t_{n}^{4}) = \frac{t_n^4}{4}|u_{n}|_{2}^{2}+\frac{t_n^4}{16}|u_{n}|_{2}^{4}
     -\frac{3-p}{2p}t_n^4 |u_{n}|_{p}^{p}
&= \frac{t_{n}^{2}}{4}|v_{n}|_{2}^{2}
    +\frac{1}{16}|v_{n}|_{2}^{4}-\frac{3-p}{2p}t_{n}^{6-2p}|v_{n}|_{p}^{p}\\
   &\geq \frac{t_{n}^{2}}{4}|v_{n}|_{2}^{2}+\frac{1}{16}
    |v_{n}|_{2}^{4}-\frac{3-p}{2p}C_{4}t_{n}^{6-2p}|v_{n}|_{2}^{2}.
\end{align*}
Consequently,
\begin{equation}\label{3-12}
    |v_{n}|_{2}=
    \begin{cases}
    o(t_{n}^{1/2})      \qquad &\text{if}\ p=3,\\
    o(t_{n}^{(3-p)/2})  &\text{if}\ 2<p<3.
    \end{cases}
\end{equation}
Moreover, by assumption we also have that
\begin{equation*}
o(1)  = t_n^4 J(u_n)= t_n^4 \Bigl( 2|\nabla u_n|_2^2 + |u_{n}|_{2}^{2}+V_{0}(u_{n})-\frac{|u_{n}|_{2}^{4}}{4}-\frac{2p-2}{p}|u_{n}|_{p}^{p}\Bigr).
\end{equation*}
Combining this with (\ref{eq:normalized}), (\ref{3-12}) and the fact that
$$
V_{0}(u_{n})= t_n^4 \int_{\mathbb{R}^{2}}
     \int_{\mathbb{R}^{2}}\log \left(|t_n x-t_n y|\right)u_n^{2}(t_n x)u_n^{2}(t_n y)dxdy = t_n^{-4}\Bigl(V_0(v_n) +|v_n|_2^4\log t_n \Bigr)
$$
we infer that
\begin{align}
o(1) & = 2 + t_{n}^{2}|v_{n}|_{2}^{2}+V_{0}(v_{n})+
   |v_{n}|_{2}^{4}\log t_{n} - \frac{1}{4}|v_{n}|_{2}^{4}-\frac{2p-2}{p}
   t_{n}^{6-2p}|v_{n}|_{p}^{p} \nonumber\\
&=2+V_{0}(v_{n})+ o(1). \label{result:G-t_n}
\end{align}
Since $V_0 = V_1-V_2$, it now follows from (\ref{2-2}), (\ref{3-12}), (\ref{result:G-t_n}) and the Gagliardo-Nirenberg inequality that
\begin{equation*}
    2 + V_{1}(v_{n}) = V_{2}(v_{n})+o(1)
     \leq C_{0}|v_{n}|_{\frac{8}{3}}^{4}+o(1)
     \leq C_{5}|v_{n}|_{2}^{3}+o(1)=o(1).
\end{equation*}
Since $V_1$ is nonnegative, this is a contradiction. We thus conclude that \eqref{3-8} holds, and
using \eqref{3-2-1} again we may then deduce that
\begin{equation*}
     \frac{1}{4} |u_{n}|_{2}^{4} +  \frac{p-2}{p}|u_{n}|_{p}^{p}
     = |\nabla u_{n}|_{2}^{2} + I'(u_{n})u_{n} -  J(u_{n}) \leq C_{3} + o(1).
\end{equation*}
Consequently, $(u_{n})_n$ is bounded in $L^{2}(\mathbb{R}^{2})$, and together with (\ref{3-8}) this shows that $(u_{n})_n$ is bounded in $H^{1}(\mathbb{R}^{2})$, as claimed.
\end{proof}

We now define, for a function $u \in \mathbb{R}^{2} \rightarrow \mathbb{R}$
and $z \in \mathbb{R}^{2}$, the translated function
\begin{equation*}
    z \ast u\::\: \mathbb{R}^{2} \rightarrow \mathbb{R},\quad
    (z \ast u)(x)=u(x - z) \quad \text{for}\ x \in \mathbb{R}^{2}.
\end{equation*}
We then may derive the following compactness property (modulo translation) for the class of Cerami sequence satisfying (\ref{3-2-1}).
It is a variant of \cite[Proposition 3.1]{Cingolani-Weth-2016} based on Proposition~\ref{lem3-4}.

\begin{proposition}\label{prop3-5}
Let $p>2$, and let $(u_{n})_n$ be a sequence in $X$ that
satisfies \eqref{3-2-1}. Then, after passing to a subsequence,
one of the following occurs:
\begin{itemize}
\item[\rm(I)] $\|u_n\| \to 0$ and $I(u_n) \to 0$ as $n \to \infty$.
\item[\rm(II)] There exist points $y_{n} \in \mathbb{R}^{2}, n \in \mathbb{N}$
such that
\begin{equation*}
    y_{n} \ast u_{n} \rightarrow u \quad \text{strongly\ in}\ X
    \ \text{as}\ n \rightarrow \infty
\end{equation*}
for some nonzero critical point $u \in X$ of $I$.

\end{itemize}
\end{proposition}

\begin{proof}
We first note that $(u_n)_n$ is bounded in $H^1(\R^2)$ by Proposition~\ref{lem3-4}. Suppose that (I) does not hold for any subsequence of $(u_n)_n$. We then claim that
\begin{equation}\label{3-13}
    \liminf_{n\rightarrow\infty}\sup_{y \in \mathbb{R}^{2}}
    \int_{B_{2}(y)} u_{n}^{2}(x)dx>0.
\end{equation}
Assuming the contrary that \eqref{3-13} does not occur. By Lions' vanishing
lemma (see e.g. \cite{Lions-1984,Willem-1996}), after passing to a subsequence,
it follows that
\begin{equation*}
    u_{n} \rightarrow 0\quad \text{in}\  L^{s}(\mathbb{R}^{2}) \quad
    \text{for\ all}\ s>2.
\end{equation*}
Therefore, by \eqref{2-2} and \eqref{3-2-1} we have
\begin{equation*}
    \|u\|_{n}^{2} + V_{1}(u_{n})=I'(u_{n})u_{n}+V_{2}(u_{n}) + |u_{n}|_{p}^{p}
    \rightarrow 0 \qquad \text{as}\ n \rightarrow \infty.
\end{equation*}
Hence, we obtain $\|u_{n}\| \rightarrow 0$ and $V_{1}(u_{n}) \rightarrow 0$,
and thus
\begin{equation*}
    I(u_{n})=\frac{1}{2}\|u_{n}\|^{2}+\frac{1}{4}\left(V_{1}(u_{n}) -
    V_{2}(u_{n})\right)-\frac{1}{p}|u_{n}|_{p}^{p} \rightarrow 0\qquad
    \text{as}\ n \rightarrow \infty.
\end{equation*}
This contradicts our assumption that (I) does not hold for any such subsequence.
So, \eqref{3-13} holds. Going if necessary to a subsequence,
there exists a sequence $\{y_{n}\} \subset \mathbb{R}^{2}$ such that,
the sequence of the functions
\begin{equation*}
    \tilde{u}_{n}:= y_{n} \ast u_{n} \in X \quad \text{with}
    \ n \in \mathbb{N},
\end{equation*}
converges weakly in $H^{1}(\mathbb{R}^{2})$ to some function $u \in H^{1}
(\mathbb{R}^{2}) \setminus \{0\}$. Consequently, we may assume that
$\tilde{u}_{n}(x)\rightarrow u(x)$ a.e. in $\mathbb{R}^{2}$.
Moreover, using \eqref{3-2-1} again, we deduce that
\begin{equation*}
    B_{1}(\tilde{u}_{n}^{2},\tilde{u}_{n}^{2})= V_{1}(\tilde{u}_{n})
    =V_{1}(u_{n})=o(1)+V_{2}(u_{n}) + |u_{n}|_{p}^{p}-\|u_{n}\|^{2},
\end{equation*}
and the RHS of this inequality remains bounded in $n$. Thus, Lemma \ref{lem2-1}
implies that $|\tilde{u}_{n}|_{\ast}$ remains bounded in $n$, so that the
sequence $\{\tilde{u}_{n}\}$ is bounded in $X$. Then, passing to a subsequence
again if necessary, we may assume that $\tilde{u}_{n} \rightharpoonup u$
weakly in $X$, so that $u \in X$. It then follows from Lemma \ref{lem2-2}(i)
that $\tilde{u}_{n} \rightarrow u$ strongly in $L^{s}(\mathbb{R}^{2})$
for $s \geq 2$. Next, we claim that
\begin{equation}\label{3-14}
    I'(\tilde{u}_{n})(\tilde{u}_{n}-u) \rightarrow 0 \qquad \text{as}\ n\rightarrow\infty.
\end{equation}
Indeed, we have
\begin{equation}\label{3-15}
    \left|I'(\tilde{u}_{n})(\tilde{u}_{n}-u)\right|=\left|I'(u_{n})
    \left(u_{n} - (-y_{n})\ast u\right)\right| \leq \|I'(u_{n})\|_{X'}
    \left(\|u_{n}\|_{X}+\|(-y_{n})\ast u\|_{X}\right)
\end{equation}
for every $n$. Moreover, we can easily see that
\begin{equation*}
    |u_{n}|_{\ast}^{2} = \int_{\mathbb{R}^{2}} \log \left(1+|x-y_{n}|\right)
     \tilde{u}_{n}^{2}(x)dx \geq C_{1}\log \left(1+|y_{n}|\right)\qquad
     \text{for\ all}\ n
\end{equation*}
with a constant $C_{1}>0$ and
\begin{equation*}
    \left|(-y_{n})\ast u\right|_{\ast}^{2} = \int_{\mathbb{R}^{2}}
      \log \left(1+|x-y_{n}|\right)u^{2}(x)dx \leq
      C_{2}\log \left(1+|y_{n}|\right)\qquad \text{for\ all}\ n
\end{equation*}
with a constant $C_{2}>0$. Combining these two inequalities with \eqref{3-13}, we then find a constant $C_{3}>0$ such that, after passing to a subsequence,
\begin{equation}\label{3-16}
  \left\|(-y_{n})\ast u\right\|_{X}^2 = \|u\|^2 + \left|(-y_{n})\ast u\right|_{\ast}^{2}
\leq C_{3} (\|u_n\|^2 +|u_n |_{\ast}^{2})= C_3 \|u_{n}\|_{X}^2
\end{equation}
for\ all $n \in \N$. Hence \eqref{3-15} implies that
\begin{equation*}
    \left|I'(\tilde{u}_{n})(\tilde{u}_{n}-u)\right| \leq
    (1+\sqrt{C_{3}}) \|I'(u_{n})\|_{X'}\|u_{n}\|_{X}\rightarrow 0
    \qquad \text{as}\ n \rightarrow \infty,
\end{equation*}
as claimed in \eqref{3-14}. Using \eqref{3-14}, we conclude that
\begin{align*}
    o(1) &= I'(\tilde{u}_{n})(\tilde{u}_{n}-u)\\
     &=o(1)+\|\tilde{u}_{n}\|^{2}-\|u\|^{2}+\frac{1}{4}V'_{0}(\tilde{u}_{n})
      (\tilde{u}_{n}-u)-\int_{\mathbb{R}^{2}}|\tilde{u}_{n}|^{p-2}
      \tilde{u}_{n}(\tilde{u}_{n}-u)dx\\
     &=o(1)+\|\tilde{u}_{n}\|^{2}-\|u\|^{2}+\frac{1}{4}\left [V'_{1}(\tilde{u}_{n})
      (\tilde{u}_{n}-u)- V'_{2}(\tilde{u}_{n})(\tilde{u}_{n}-u)\right],
\end{align*}
where
\begin{equation*}
    \left|\frac{1}{4}V'_{2}(\tilde{u}_{n})(\tilde{u}_{n}-u)\right| = \left|B_{2}
    \left(\tilde{u}_{n}^{2}, \tilde{u}_{n}(\tilde{u}_{n}-u)\right)\right|
    \leq |\tilde{u}_{n}|_{\frac{8}{3}}^{3}\left|\tilde{u}_{n}-u
    \right|_{\frac{8}{3}}\rightarrow 0 \qquad \text{as}\ n\rightarrow\infty,
\end{equation*}
and
\begin{equation*}
    \frac{1}{4}V'_{1}(\tilde{u}_{n})(\tilde{u}_{n}-u) = B_{1}\left(\tilde{u}_{n}^{2},
    \tilde{u}_{n}(\tilde{u}_{n}-u)\right)=B_{1}\left(\tilde{u}_{n}^{2},
    (\tilde{u}_{n}-u)^{2}\right)+B_{1}\left(\tilde{u}_{n}^{2},
    u(\tilde{u}_{n}-u)\right)
\end{equation*}
with
\begin{equation*}
    B_{1}\left(\tilde{u}_{n}^{2}, u(\tilde{u}_{n}-u)\right) \rightarrow 0
    \qquad \text{as}\ n\rightarrow\infty
\end{equation*}
by Lemma \ref{lem2-7}. Combining these estimates, we infer that
\begin{equation*}
    o(1) = \|\tilde{u}_{n}\|^{2}-\|u\|^{2}+ B_{1}\left(\tilde{u}_{n}^{2},
    (\tilde{u}_{n}-u)^{2}\right)+o(1)\geq\|\tilde{u}_{n}\|^{2}-\|u\|^{2}+o(1),
\end{equation*}
which implies that $\|\tilde{u}_{n}\| \rightarrow \|u\|$ and
$B_{1}\left(\tilde{u}_{n}^{2}, (\tilde{u}_{n}-u)^{2}\right) \rightarrow 0$
as $n \rightarrow \infty$. Therefore,  $\|\tilde{u}_{n}-u\|\rightarrow0$
as $n \rightarrow \infty$. Moreover, by Lemma \ref{lem2-1} we have
$|\tilde{u}_{n} - u|_{\ast}\rightarrow0$. We thus deduce that
$\|\tilde{u}_{n} - u \|_{X} \rightarrow 0$ as $n \rightarrow \infty$,
as claimed.

Finally, we need to show that $I'(u) = 0$. Let $v \in X$. Then,
by the same argument which leads to \eqref{3-16}, we obtain
\begin{equation*}
    \left\|(-y_{n})\ast v\right\|_{X} \leq C_{4} \|u_{n}\|_X \qquad
    \text{for\ all}\ n
\end{equation*}
with a constant $C_{4}>0$. So,  from \eqref{3-2-1} we deduce that
\begin{align*}
    |I'(u)v|& = \lim_{n \rightarrow \infty} |I'(\tilde{u}_{n})v|
     =\lim_{n \rightarrow \infty}\left|I'(u_{n})\left[(-y_{n})
     \ast v\right]\right| \\
    &\leq  \lim_{n \rightarrow \infty}\|I'(u_{n})\|_{X'}
     \|(-y_{n})\ast v\|_{X} \leq C_{4}\lim_{n \rightarrow \infty}
     \|I'(u_{n})\|_{X'}\|u_{n}\|_{X} =0.
\end{align*}
This completes the proof.
\end{proof}

\begin{proof}[Proof of Theorem \ref{th1-1}]
By Lemma~\ref{lem3-3} and Proposition~\ref{prop3-5}, there exists a critical point $u \in X \setminus \{0\}$ of $I$ with $I(u) = c_{mp}$, which already completes the proof of Theorem~\ref{th1-1}(i). In particular, the set
\begin{equation*}
  \mathcal{K}=\{u \in X \backslash \{0\}\::\: I'(u)=0 \}
\end{equation*}
is nonempty. Let $(u_n)_n \subset \mathcal{K}$ be a sequence such that
$$
I(u_n) \to c_{g}= \inf_{u \in \mathcal{K}}I(u) \ \in [-\infty,c_{mp}].
$$
By definition of $\mathcal{K}$ and Lemma~\ref{lem2-4}, the sequence $(u_n)_n$ satisfies (\ref{3-2-1}). Moreover, by
(\ref{2-5-1}) we have
$$
\liminf_{n \to \infty} \|u_n\|\geq \rho>0
$$
Consequently, by Proposition~\ref{prop3-5} there exist, after passing to a subsequence,
points $x_{n} \in \mathbb{R}^{2}, n \in \mathbb{N}$ and a nonzero critical point $u \in X$ of $I$ such that
\begin{equation*}
    x_{n} \ast {u}_{n} \rightarrow u \quad \text{strongly\ in}
    \ X \ \text{as}\ n \rightarrow \infty
\end{equation*}
So $u \in \mathcal{K}$ and
$$
I (u)=\lim_{n \to \infty} I(x_{n} \ast {u}_{n}) = \lim_{n \to \infty} I({u}_{n})=c_{g}.
$$
In particular we have $c_{g}>-\infty$, and $u$ has the properties asserted in Theorem \ref{th1-1}(ii).
\end{proof}

\section{Proof of Theorem \ref{th1-1-1}}
\label{sec:proof-theorem-refth1}

\indent

In this section, we will give the proof of Theorem \ref{th1-1-1}. For this we need to elaborate the saddle point structure of the functional $I$ in the case $p \ge 3$. We need the following lemma.
\begin{lemma}\label{lem4-1}
Let $C_{i} \in \R$, $C_i>0$ for $i=1,3,4$, and let $C_{2} \in \R$. If $p\geq3$, then the function
\begin{equation*}
f:(0,\infty) \to \R,\qquad f(t)=C_{1}t^{2}+C_{2}t^{4}-C_{3}t^{4}\log t-C_{4}t^{2p-2}
\end{equation*}
has a unique positive critical point $t_0$ such that $f'(t)>0$ for $t<t_0$ and $f'(t)>0$ for $t>0$.
\end{lemma}

\begin{proof}
The proof is elementary, so we omit it.
\end{proof}

Similarly as in \cite{Ruiz-2006}, we now consider the functional $J$ defined in (\ref{eq:def-G}) and set $\mathcal{M}$ defined in (\ref{eq:def-M}), i.e.,
\begin{equation*}
    \mathcal{M}=\left\{u \in X \backslash \{0\}: J(u)=0\right\},
\end{equation*}
We then have
\begin{equation}\label{4-1}
    J(u)=2I'(u)u-P(u),
\end{equation}
where $P(u)$ is given in Lemma~\ref{lem2-4}. As already noted in the introduction, it follows from Lemma~\ref{lem2-4} that every critical point of $I$ is contained in $\mathcal{M}$: Indeed this is true for arbitrary $p>2$.

In the following, for $u\in X$ and $t>0$, we set $Q(t,u):=u_t \in X \backslash \{0\}$, i.e.,
$$
Q(t,u)(x):= u_t(x)=t^{2}u(tx) \qquad \text{for $x \in \R^2$.}
$$

\begin{lemma}\label{lem4-2}
Let $p \geq 3$.
\begin{itemize}
\item[\rm(i)] For any $u \in X \backslash \{0\}$,
there exists a unique $t_u > 0$ such that
$Q(t_u, u)  \in \mathcal{M}$.
\item[\rm(ii)] For any $u \in X \backslash \{0\}$, $t_u$ is the unique maximum point of the function $(0,\infty) \to \R$, $t \mapsto
I(Q(t,u))$.
\item[\rm(iii)] The map $X \backslash \{0\} \to (0,\infty)$, $u \mapsto t_u$ is continuous.
\item[\rm(iv)] Every $u \in \mathcal{M}$ with $I(u)=c_{\mathcal M}$ is a critical point of $I$ which does not change sign on $\R^2$.
\end{itemize}
\end{lemma}

\begin{proof}
For $u \in X \setminus \{0\}$, consider the function $h_u: (0,\infty) \to \R$, $
h_u(t):=I(Q(t,u))$. As in \eqref{comp-rho-s-v}, we see that
\begin{align}
  h_u(t)&=\frac{t^{4}}{2}\int_{\mathbb{R}^{2}}|\nabla u|^{2}dx
    +\frac{t^{2}}{2}\int_{\mathbb{R}^{2}}u^{2}dx+\frac{t^{4}}{4}
    \int_{\mathbb{R}^{2}}\int_{\mathbb{R}^{2}}
    \log\left(|x-y|\right)u^{2}(x)u^{2}(y)dxdy \label{eq:def-h}\\
  &\quad-\frac{t^{4}\log t}{4}\left(\int_{\mathbb{R}^{2}}| u|^{2}dx\right)^{2}
    -\frac{t^{2 p-2}}{p}\int_{\mathbb{R}^{2}}|u|^{p}dx.\nonumber
\end{align}
By Lemma \ref{lem4-1}, $h$ has a unique critical point
$t_u>0$ such that
\begin{equation}
  \label{eq:monotonicity}
h_u'(t)>0 \quad \text{for $t \in (0,t_u)$} \qquad \text{and}\qquad h_u'(t)<0 \quad \text{for $t >t_u$.}
\end{equation}
Similarly as in (\ref{comp-partial-s}), it also follows that $h_u'(t)=  \frac{J(Q(t,u))}{t}$ for $t>0$.
This gives (i) and (ii). Combining (\ref{eq:monotonicity}) with the fact that the map $X \setminus \{0\} \to \R, \:
u \mapsto h_u'(t)$ is continuous for fixed $t>0$, we also deduce that the map $X \backslash \{0\} \to (0,\infty)$, $u \mapsto t_u$ is continuous, as claimed in (iii).

It thus remains to prove (iv). Let $u \in \mathcal M$ be an arbitrary minimizer for $I$ on
$\mathcal M$. To show that $u$ is a critical point of $I$, we argue
by contradiction and assume that there exists $v \in X$
such that $I'(u)v < 0$. Since $I$ is a $C^1$-functional on $X$, we may then fix $\eps>0$ with
the following property: \medskip

{\em For every $\tau \in (0,\eps)$, every $w \in X$ with
  $\|w\|_{X}<\eps$ and every $\tilde v \in X$ with $\|\tilde v - v\|_X<\eps$
we have
  \[I(u+w+\tau \tilde v) \le I(u +w)-\eps \tau.\]}
\hspace{-1ex}Using (iii) and the fact that $t_u=1$, we may then choose $\tau \in (0,\eps)$ sufficiently small such that for $t^\tau:= t_{u+\tau v}$ we have
$$
|t^\tau-1| <\frac{\eps}{\|u\|_X} \qquad \text{and}\qquad  \|Q(t^\tau,v)- v\|_X < \eps.
$$
Setting $w:= Q(t^\tau,u)-u$ and $\tilde v:= Q(t^\tau,v)$, we then have $\|w\|_{X}<\eps$ and $\|\tilde v - v\|_X<\eps$, so that, by the property above,
\begin{align*}
I\bigl(Q(t^\tau,u + \tau v)\bigr)= I\bigl(Q(t^\tau,u) + \tau Q(t^\tau,v)\bigr)&= I(u + w + \tau \tilde v) \le I(u+w) - \eps \tau \\
& < I(u+w)= I(Q(t^\tau,u)) \le I(u)= c_{\mathcal M}.
\end{align*}
Since $Q(t^\tau,u + \tau v) \in \mathcal{M}$, this contradicts the definition of $c_{\mathcal M}$.
Hence $u$ is a critical point of $I$.

Finally, to see that $u$ does not change sign, we note that $I(u)=I(|u|)$ and $J(u)=J(|u|)$, so that $|u|$ is a minimizer of $I|_{\mathcal{M}}$ as well. Hence $|u|$ is a critical point of $I$ by the considerations above. By Remark \ref{sec:regularity},
$|u| \in C^{2}(\mathbb{R}^{2})$ and $-\Delta |u| + q |u| = 0$ in $\mathbb{R}^{2}$
with some function $q \in L_{loc}^{\infty}(\mathbb{R}^{2})$.
Hence, the strong maximum principle and the fact that $u \neq 0$ implies that
$|u|>0$ in $\mathbb{R}^{2}$, which shows that $u$ does not change sign.
 \end{proof}

\begin{lemma}\label{lem4-3}
Let $p \geq 3$. For the energy values $c_g, c_{\mathcal M}, c_{mm}$ and $c_{mp}$ defined in the introduction, we then have $c_g= c_{\mathcal M} = c_{mm}= c_{mp}$.
\end{lemma}

\begin{proof}
By (\ref{eq:eq-comp-c-g-c-M}) we have $c_{\mathcal M}\le c_g \le c_{mp}$, and by Lemma~\ref{lem4-2} we have $c_{\mathcal M} = c_{mm}$.
Moreover, from Lemma~\ref{lem2-6} we deduce that $c_{mp} \le c_{mm}$. Thus the claim follows.
\end{proof}

The {\em proof of Theorem~\ref{th1-1-1}} is now completed by merely combining Theorem~\ref{th1-1},  
Lemma~\ref{lem4-2} and Lemma~\ref{lem4-3}.

\section{The symmetric setting}

\indent

This section is devoted to the proofs of Theorem \ref{th1-2} and Corollary~\ref{coro1-3}.
Some parts of the proof of Theorem \ref{th1-2} are similar to the proof of Theorem \ref{th1-1} and will therefore only be outlined. From now on, we fix a closed subgroup $G$ of the orthogonal group $O(2)$, and we let $\tau\::\: G\rightarrow \{-1, 1\}$ be a group homomorphism. We also consider the action $\ast$ of $G$ on $X$ defined by (\ref{1-7}), and we assume that the corresponding invariant subspace
\begin{equation*}
    X_{G}:=\{u \in X\::\: A \ast u= u\ \text{for\ all}\ A \in G\},
\end{equation*}
is not the null space. Our aim is
to detect critical points of the restriction of
the functional $I$ to $X_G$.
By the principle of symmetric criticality (see \cite[Theorem 1.28]
{Willem-1996}), any critical point of the restriction of $I$
to $X_{G}$ (which we will denote by $I$ as well in the following)
is a critical point of $I$.

From Lemmas \ref{lem2-5} and \ref{lem2-6}, we easily deduce that
\begin{equation}
\label{def-mp-G}
  0 < m_{\rho} \leq c_{mp,G}:=\inf\limits_{\gamma \in \Gamma_G}
   \max_{t \in [0,1]}I(\gamma(t))< \infty,
\end{equation}
where $\Gamma_G=\left\{\gamma \in C\left([0,1],\: X_{G}\right)\ | \ \gamma(0)=0,
I(\gamma(1))<0\right\}$. Similar as in the proof of Lemma \ref{lem3-3}, we then deduce from Proposition \ref{prop3-2} that
there exists a sequence $\{u_{n}\}\subset X_{G}$ such that,
as $n\rightarrow\infty$,
\begin{equation}\label{5-1}
  I(u_{n})\rightarrow c_{mp,G},
  \quad \|I'(u_{n})\|_{X_{G}'}\left(1+\|u_{n}\|_{X}\right)\rightarrow 0
  \quad \text{and} \quad J(u_{n}) \rightarrow 0,
\end{equation}
where $J$ is defined in (\ref{eq:def-G}).
The invariance of $I$ under the action $G$ implies that $I'(v)w=0$ for all
$v \in X_{G}$ and $w \in X_{G}^{\perp}$, where $X_{G}^{\perp}$ denotes the
orthogonal complement of $X_{G}$ on $X$. Therefore, we have
$\|I'(v)\|_{X_{G}'}=\|I'(v)\|_{X'}$ for all $v \in X_{G}$.
Consequently, we may rewrite \eqref{5-1} as
\begin{equation}
\label{5-1-1}
  I(u_{n})\rightarrow c_{mp,G},
  \quad \|I'(u_{n})\|_{X'}\left(1+\|u_{n}\|_{X}\right)\rightarrow 0
  \quad \text{and} \quad  J(u_{n}) \rightarrow 0.
\end{equation}
We also have the following variant of Proposition~\ref{prop3-5}. Here we consider the subspace
$$
\Fix(G):= \{x \in \R^2\::\: Ax = x \quad \text{for all $A \in G$}\} \; \subset\; \R^2.
$$

\begin{proposition}\label{prop3-5-1}
Let $p>2$, and let $(u_{n})_n$ be a sequence in $X_G$ that
satisfies
\begin{equation}\label{3-2-1-1}
c:= \sup_{n \in \N}  I(u_{n})<\infty
  \qquad \text{and} \qquad \|I'(u_{n})\|_{X'}\left(1+\|u_{n}\|_{X}\right)\rightarrow 0, \quad J(u_{n}) \rightarrow 0 \quad \text{as $n \to \infty$.}
\end{equation}
Then, after passing to a subsequence,
one of the following occurs:
\begin{itemize}
\item[\rm(I)] $\|u_n\| \to 0$ and $I(u_n) \to 0$ as $n \to \infty$.
\item[\rm(II)] There exist points $y_{n} \in \Fix(G)$, $n \in \mathbb{N}$
such that
\begin{equation*}
    y_{n} \ast u_{n} \rightarrow u \quad \text{strongly\ in}\ X
    \ \text{as}\ n \rightarrow \infty
\end{equation*}
for some nonzero critical point $u \in X_G$ of $I$.
\end{itemize}
\end{proposition}

\begin{proof}
Suppose that $(I)$ does not hold. By Proposition~\ref{prop3-5}, we may pass to a subsequence such that
\begin{equation*}
    \tilde y_{n} \ast u_{n} \rightarrow \tilde u \quad \text{in}\ X \
    \text{as}\ n \rightarrow \infty
\end{equation*}
for suitable points
$\tilde y_{n} \in \mathbb{R}^{2}, n \in \mathbb{N}$ and some nonzero critical point $ \tilde u \in X$ of $I$. We first claim that
\begin{equation}
  \label{eq:supAG}
\sup \{ |A \tilde y_n -\tilde y_n|  \::\: A \in G,\: n \in \N\} < \infty.
\end{equation}
To see this, we set $z_n:=
-\tilde y_n$ for $n \in \N$, and we let $(A_n)_n \subset G$ be an arbitrary sequence. Since $A_n* u_n = u_n$ for every $n
\in \N$, we have
\begin{align}
|A_n * (z_n * \tilde u)- z_n * \tilde u |_2 \le  |A_n* ( z_n * \tilde u- u_n) |_2&+ |u_n- z_n *
\tilde u|_2= 2 |z_n* \tilde u - u_n|_2 \nonumber\\
=& 2 |\tilde u- \tilde y_n * u_n|_2  \to 0 \qquad \text{as $n \to \infty$.}    \label{eq:49}
\end{align}
Setting $v_n:=A_n* \tilde u$ and $\zeta_n:=A_n z_n-z_n$ for $n \in \mathbb N$, we also find that
\begin{align}
\bigl|A_n  * (z_n *\tilde u)- &z_n * \tilde u\bigr|_2^2 - 2|\tilde u|_2^2=
-2\tau(A_n) \int_{\R^2}\tilde u(A_n^{-1}
x-z_n)\tilde u(x-z_n)\,dx \nonumber\\
&=-2\tau(A_n) \int_{\R^2}\tilde u(A_n^{-1}(y -
\zeta_n))\tilde u(y)\,dy = -2  \int_{\R^2} (\zeta_n * v_n) \tilde u \,dy \label{eq:50}
\end{align}
Since $G$ is compact as a closed subgroup of $O(2)$, we may pass to a subsequence such that $A_n \to A \in G$ as $n \to \infty$, which implies that $v_n \to v:= A * \tilde u$ in $X$ and therefore
\begin{equation}
  \label{eq:zeta-n-3}
|\zeta_n * (v_n-v)|_2 = |v_n-v|_2 \to 0 \qquad \text{as $n \to \infty$.}
\end{equation}
Combining (\ref{eq:49}), (\ref{eq:50}) and (\ref{eq:zeta-n-3}) gives
\begin{equation}
  \label{eq:key-limit}
\lim_{n \to \infty} \int_{\R^2} (\zeta_n * v) \tilde u \,dz = |\tilde u|_2^2 >0.
\end{equation}
From this we deduce that $(\zeta_n)_n$ remains bounded, since otherwise $\zeta_n * v \rightharpoonup 0$ in $L^2(\R^2)$ after passing to a subsequence.
Since
$$
|A_n \tilde y_n -\tilde y_n|= |A_n z_n-z_n|= |\zeta_n| \qquad \text{for $n \in \mathbb N$,}
$$
we thus infer that also $A_n \tilde y_n-\tilde y_n$ remains bounded, and this gives (\ref{eq:supAG}).

We now replace $\tilde y_n$ by
$$
y_n:= \frac{1}{\mu(G)} \int_{G} A \tilde y_n\,d\mu(A) \; \in\; \Fix(G),\qquad n \in \mathbb N,
$$
where $\mu$ denotes the Haar measure of $G$. By (\ref{eq:supAG}) we infer that $y_n- \tilde y_n$ remains bounded in $\R^2$ as $n \to \infty$, and thus we may pass to a subsequence such that $y_n - \tilde y_n \to r \in \R^2$.
Consequently,
\begin{equation*}
   y_{n} \ast u_{n} \rightarrow u := r * \tilde u \quad \text{in}\ X \
    \text{as}\ n \rightarrow \infty
\end{equation*}
Finally, for $A \in G$ and $x \in \R^2$ we have, since $y_n \in \Fix(G)$,
\begin{align*}
[A * u](x) &= \lim_{n \to \infty} A* \Bigl( y_{n} \ast u_{n}\Bigr) = \tau(A) \lim_{n \to \infty} u_n (A^{-1}x- y_n)=
\tau(A) \lim_{n \to \infty} u_n (A^{-1}(x- y_n)) \\
&=\lim_{n \to \infty} [y_n * (A* u_n)](x)= \lim_{n \to \infty} [y_n * u_n](x) = u(x).
\end{align*}
We thus conclude that $u \in X_G$. This shows that alternative $(II)$ holds, as claimed.
\end{proof}

\begin{proof}[Proof of Theorem~\ref{th1-2} (completed)]
Using Proposition~\ref{prop3-5-1}, we may now complete the proof of Theorem~\ref{th1-2} precisely as we completed the proof of Theorem~\ref{th1-1} in the last part of Section~\ref{sec:exist-ground-state}, replacing $X$ by $X_G$ and $c_{mp}$ by $c_{mp,G}$.
\end{proof}

It thus remains to give the proof of Corollary \ref{coro1-3}.

\begin{proof}[Proof of Corollary \ref{coro1-3}]
Similarly as in Example~\ref{exam1-1}(iii) we consider, for $n \in \N$, the subgroup $G_n$ of $O(2)$
of order $2 \cdot 3^{n}$ generated by the (counter-clockwise) $\frac{\pi}{3^n}$-rotation
\begin{equation*}
    A_n \in O(2), \quad A_n x=\left(x_{1}\cos \frac{\pi}{3^n}-x_{2}
    \sin \frac{\pi}{3^n},\ x_{1}\sin\frac{\pi}{3^n}+ x_{2}
    \cos\frac{\pi}{3^n}\right) \quad \text{for}\
    x=(x_{1}, x_{2}) \in \mathbb{R}^{2}.
\end{equation*}
Let $\tau_n: G_n\rightarrow \{-1, 1\}$ be the homomorphism defined by
\begin{equation*}
    \tau_n(A_n^{j}) =(-1)^{j} \quad \text{for}\ j=1,\cdot\cdot\cdot,2 \cdot 3^{n},
\end{equation*}
where $A_n^{j}$ is the $\frac{j\pi}{3^n}$-rotation. By definition, we have that
\begin{equation*}
X_{G_{n+1}} \subset X_{G_{n}} \qquad \text{for $n \in \N$.}
\end{equation*}
Consequently, we also have that
$$
c_{mp,G_{n+1}} \ge c_{mp,G_n} \ge c_{mp,G_1} >0 \qquad \text{for all $n \in \N$.}
$$
Then Theorem \ref{th1-2} applies and yields a nonradial sign-changing solution $u_n \in X_{G_n}$ with $I(u_n)=c_{mp,G_n}$ for every $n \in \N$.
Suppose by contradiction that
$$
\limsup_{n \to \infty}  I(u_{n})<\infty.
$$
Applying Proposition~\ref{prop3-5-1} for fixed $n \in \N$ and using the fact that $\Fix(G_n)=\{0\}$ for all $n \in \N$, we may now pass to a subsequence such that
$$
u_n \to u \qquad \text{in $X$,}
$$
where
\begin{equation}
  \label{eq:intersection-symmetry}
u \in \bigcap_{n \in \mathbb N}X_{G_{n}},
\end{equation}
and $u \in X \setminus \{0\}$ is a critical point of $I$. As noted in Remark~\ref{sec:regularity}, it then follows that $u$ is of class $C^{2}$ on $\R^2$.
From (\ref{eq:intersection-symmetry}) we then deduce that
$$
u(x)= - u(A_n x ) \qquad \text{for all $x \in \R^2$, $n \in \N$.}
$$
Since $A_n x \to x$ for $x \in \R^2$ as $n \to \infty$, it follows that $u(x)= 0$ for $x \in \R^2$. This is a contradiction, and thus the claim follows.
\end{proof}

\vskip 0.6 true cm

\noindent\textbf{\Large Acknowledgements}
\vskip 0.2 true cm

This work was completed when the first author visited Goethe-University
of Frankfurt in 2016-2017, and he would like to thank Tobias Weth for
his hospitality. The first author was partially  supported by Natural Science Foundation
of China (Nos. 11171135, 51276081, 11571140, 11671077, 11601204) and Major Project 
of Natural Science Foundation of Jiangsu Province Colleges and Universities (No. 14KJA100001).

\vskip 0.6 true cm

\end{document}